\documentclass[12pt]{article}
\usepackage{amsmath,amssymb,amsthm, amsfonts}
\usepackage{etaremune, enumerate, float, verbatim}
\usepackage{hyperref}
\usepackage{graphicx}
\usepackage{tikz}
\usetikzlibrary{backgrounds}
\usetikzlibrary{shapes.geometric}
\usepackage{fancyhdr}
\usepackage{graphicx, subcaption, caption}
\usepackage[algo2e,ruled,vlined]{algorithm2e}
\usepackage{color}
\usepackage{url}
\usepackage[mathlines]{lineno}
\usepackage{graphicx, subcaption, caption}
\usepackage[algo2e,ruled,vlined]{algorithm2e}
\usepackage{color}
\definecolor{red}{rgb}{1,0,0}

\definecolor{blue}{rgb}{0,.3,1}

\definecolor{green}{rgb}{0,.8,0} %
\def\gre{\color{green}}
\definecolor{purp}{rgb}{.5,0,.5}

\numberwithin{figure}{section}   

\tikzstyle{vertex}=[circle, draw=black, fill=white, thick, inner sep=0pt, minimum size=6pt]
\tikzstyle{Bvertex}=[circle, black, fill, draw=black, inner sep=0pt, minimum size=6pt]
\tikzstyle{vtx}=[circle, white, fill=white, draw=black, thick, inner sep=0pt, minimum size=6pt]
\tikzstyle{gvertex}=[circle, green, fill, draw=black, inner sep=0pt, minimum size=6pt]
\newcommand{\vertex}{\node[vertex]}
\newcommand{\vtx}{\node[vtx]}
\newcommand{\Bvertex}{\node[Bvertex]}
\newcommand{\gvertex}{\node[gvertex]}

\tikzset{
triangle/.style={inner sep=1.2pt, outer sep=0pt, regular polygon, regular polygon sides=3, draw, fill=gray, thick, minimum size=6pt},
square/.style={inner sep=1.7pt, outer sep=0pt, regular polygon, regular polygon sides=4, draw, fill=black, thick, minimum size=6pt},
}

\newtheorem{thm}{Theorem}[section]
\newtheorem{cor}[thm]{Corollary}
\newtheorem{lem}[thm]{Lemma}
\newtheorem{prop}[thm]{Proposition}

\newtheorem{obs}[thm]{Observation}

\theoremstyle{definition}
\newtheorem{rem}[thm]{Remark}

\theoremstyle{definition}
\newtheorem{defn}[thm]{Definition}

\theoremstyle{definition}
\newtheorem{ex}[thm]{Example}

\newcommand{\ZZ}{\mathbb{Z}}

\newcommand{\F}{\mathcal F}

\newcommand{\pd}{\gamma_P}

\newcommand{\thpd}{\operatorname{th}_{\rm pd}}
\newcommand{\thpdx}{\operatorname{th}_{\rm pd}^{\x}}

\newcommand{\ptpd}{\operatorname{pt}_{\rm pd}}

\newcommand{\ppt}{\operatorname{pt}_{\rm pd}}
\newcommand{\rd}{\operatorname{rd}}

\newcommand{\x}{\times}

\newcommand{\bit}{\begin{itemize}}
\newcommand{\eit}{\end{itemize}}
\newcommand{\ben}{\begin{enumerate}}
\newcommand{\een}{\end{enumerate}}
\newcommand{\beq}{\begin{equation}}
\newcommand{\eeq}{\end{equation}}
\newcommand{\bea}{\begin{eqnarray*}} %
\newcommand{\eea}{\end{eqnarray*}}
\newcommand{\bpf}{\begin{proof}}
\newcommand{\epf}{\end{proof}\ms}
\newcommand{\bmt}{\begin{bmatrix}}
\newcommand{\emt}{\end{bmatrix}}
\newcommand{\ms}{\medskip}

\newcommand{\lc}{\left\lceil}
\newcommand{\rc}{\right\rceil}
\newcommand{\lf}{\left\lfloor}
\newcommand{\rf}{\right\rfloor}

\newcommand{\noi}{\noindent}

\title{Product throttling for power domination}
\author{Sarah E. Anderson\thanks{Department of Mathematics, University of St. Thomas,
St. Paul, MN 55105, USA
(ande1298@stthomas.edu).} \and
Karen L. Collins\thanks{Department of Mathematics and Computer Science, Wesleyan University, Middletown, CT 06459, USA (kcollins@wesleyan.edu).}\and
Daniela Ferrero\thanks{Department of Mathematics, Texas State University, San Marcos, TX 78666, USA (dferrero@txstate.edu).}\and
Leslie Hogben\thanks{Department of Mathematics, Iowa State University,
Ames, IA 50011, USA and American Institute of Mathematics, %
San Jose, CA 95112, USA
(hogben@aimath.org).}\and
Carolyn Mayer \thanks{Department of Mathematical Sciences, Worcester Polytechnic Institute,
Worcester, MA 01609, USA and Sandia National Laboratories, Albuquerque, NM 87185, USA
(cdmayer@sandia.gov).}\and
Ann N. Trenk \thanks{Department of Mathematics, Wellesley College,
Wellesley, MA 02481, USA
(atrenk@wellesley.edu).}\and
Shanise Walker\thanks{Department of Mathematics, University of Wisconsin-Eau Claire,
Eau Claire, WI 54701, USA (walkersg@uwec.edu).}}

\begin{document}
\maketitle

\vspace{-10pt}

\begin{abstract}
 The product power throttling number of a graph is defined  to study product throttling for power domination. The domination number of a graph is an upper bound for its product power throttling number. %
 It is established that the two parameters are equal for certain families including paths, cycles, complete graphs, %
 unit interval graphs, and grid graphs (on the plane, cylinder, and torus).  
 Families of graphs for which the product power throttling number is less than the domination number are also exhibited.  Graphs with extremely high or low product power throttling number are characterized and bounds on the product power throttling number %
 are established.  \end{abstract}

\noi {\bf Keywords} Power domination, throttling, product throttling, product power throttling, graph searching

\noi{\bf AMS subject classification} 05C69, 05C57, 68R10, 91A43, 94C15 %

\section{Introduction}\label{sintro}

Many graph search processes observe all vertices starting with an initial set of vertices through a process consisting of rounds or discrete time steps.  
Throttling minimizes the sum or product of the resources used to accomplish a task (number of initial vertices) and the time (number of rounds) needed to complete that task. 
 Many of the graph parameters for which throttling has been studied arose from applications.  One such parameter is the power domination number, which originated from the problem of optimal placement  of Phasor Measurement Units (PMUs)  to monitor an electric power network at minimum cost.

The power domination problem was modeled using graphs by Haynes et al.~in
\cite{HHHH02}; Brueni and Heath \cite{BH} showed that a simplified version of the propagation rules  is equivalent to the original version in \cite{HHHH02}, and we use their propagation rules. 
Let $G$ be a graph and let $S$ be a non-empty subset of vertices of $G$; $N[S]$ denotes the closed neighborhood of $S$. Define the sequences of sets $P^{(i)}(S)$ and $P^{[i]}(S)$ by the following recursive rules:  \vspace{-3pt}
\ben[(1)] 
\item\label{domstep} $P^{[0]}(S)=P^{(0)}(S)=S$, $P^{[1]}(S)=N[S]$ and $P^{(1)}(S)=N[S]\setminus S$.  \vspace{-3pt}
\item\label{zfproc} For $i\ge 1$,\vspace{-3pt}
\bea P^{(i+1)}(S)&=& \big \{w\in V(G)\setminus P^{[i]}(S) : \exists u\in P^{[i]}(S), \,N_G(u)\setminus P^{[i]}(S)=\{w\} \big \},\\ \vspace{-3pt}
P^{[i+1]}(S)&=&P^{[i]}(S)\cup P^{(i+1)}(S).\vspace{-3pt}
\eea
\een\vspace{-6pt}

For $v\in  P^{(k)}(S)$, we say \emph{$v$ is observed in round $k$}.
If every vertex is observed in some round,  then $S$ is a {\it power dominating set} of $G$. The \emph{power domination number} of   $G$, denoted by $\pd(G)$, is the minimum cardinality of a power dominating set. 
 When $S$ is a power dominating set, the least positive integer $t$ with the property that $P^{[t]}(S)=V(G)$ is the {\it power propagation time} of $S$ in $G$, denoted by $\ppt(G;S)$; if $S$ is not a power dominating set, then $\ppt(G;S)=\infty$. We require $t$ to be positive because we adopt the perspective that  step \eqref{domstep} of power domination always occurs, so $\ptpd(G;S)\ge 1$ for every $S$, including $S=V(G)$.\footnote{In the original definition of power propagation time in \cite{FHKY17}, $\ptpd(G;V(G))= 0$.}  For $k\in\ZZ^+$, $\ppt(G,k)=\min_{|S|=k}\ppt(G;S)$ and the \emph{power propagation time of $G$} is $\ppt(G)=\ppt(G,\pd(G))$.  

The large scale deployment of wide area measurement systems of PMUs started in 2010 and continues growing \cite{SW18}. The analysis of available systems has shown that minimizing the number of PMUs alone yields unsatisfactory state estimation, primarily due to the loss of information in the event of transmission failures \cite{Sun19}. Since failures are inevitable, the proposed solution is to add redundancy \cite{PB18, SW18}. While higher levels of redundancy imply larger numbers of PMUs, which result in increased costs, it has been observed that adding even a few redundant PMUs has a number of advantages that offsets the cost increase \cite{PB18}. As a result, nowadays the PMU placement problem seeks a compromise between the cost of adding redundancy and the improvements in the upgraded system. In terms of power domination, this new approach to the PMU placement problem 
creates the need to study properties of the graph propagation process associated with a power dominating set in addition to its cardinality, as minimum power dominating sets might {no longer} correspond to the best choice of PMU placements. In this work {we study a combination of the number of  PMUs and the number of rounds in the power domination propagation process}, using a parameter that has proven successful in other forms of graph searching. 

Throttling sums was studied first and has been studied more widely than throttling products. Brimkov et al.~defined the (sum) power domination throttling number in \cite{powerdom-throttle}.  In this paper we introduce product throttling for power domination, establish bounds, provide conditions sufficient to guarantee the product power throttling number equals the domination number, and show that these parameters are equal for various families of graphs. %

\begin{defn} Let $G$ be a graph.  For a set $S\subseteq V(G)$, $\thpdx(G;S)=|S|\ptpd(G;S)$.    The {\em product power throttling number} of $G$ is
\[\thpdx(G)=\min_{S\subseteq V(G)}\thpdx(G;S)=\min_{S\subseteq V(G)}|S|\ptpd(G;S).\]
For $k\in\ZZ^+$, $\thpdx(G,k)=\min_{|S|=k}\thpdx(G;S)$.  \end{defn}

The product power throttling number, $\thpdx(G)$,  and the (sum) power domination throttling number, $\thpd(G):=\min_{S\subseteq V(G)}|S|+\ptpd(G;S)$, are noncomparable. %
For $K_n$, one vertex observes all vertices in one round, so $\thpdx(K_n)=1$, whereas $\thpd(K_n)=2$.   From Proposition \ref{p:pathcycle} below, $\thpdx(P_n)=\lc\frac n 3 \rc$, whereas  $\thpd(P_n)=\lc \sqrt {2n}-\frac 1 2\rc$ \cite{powerdom-throttle}.

 A main theme of this work is that for many graphs $\thpdx(G)=\gamma(G)$.  Graph families for which this is established include paths and cycles (Section \ref{s:thdpx=gamma}), unit interval graphs (Section \ref{s:unit-interval}), and Cartesian products of complete graphs with complete graphs, and of path or cycles with paths or cycles (Section \ref{s:Cart-prod}). We also  characterize connected graphs of order $n$ having $\thpdx(G)=1,2,$ and $\frac n 2$ in Section \ref{s:extreme};  Section \ref{s:bounds} contains preliminary results.

In the remainder of this introduction we present additional terminology and make some elementary observations.
  Note that $P^{(k)}(S)$ is the set of vertices that are first observed in round $k$, and the sets $P^{(0)}(S),P^{(1)}(S),\dots,$ $P^{(\ptpd(G;S))}(S)$ partition the vertices of  $G$ when $S$ is a power dominating set of $G$. For each $v \in V(G)$,  define the {\em round  function}, $\rd(v)$, to be number of the round in which vertex $v$ is first observed.  That is,    $\rd(v)=k$ for $v\in P^{(k)}(S)$.

Power domination can be thought of as a domination step \eqref{domstep} followed by a {zero forcing} process \eqref{zfproc}.  A set $S\subseteq V(G)$  \emph{dominates} a graph $G$ if $V(G)=N[S]$.  The \emph{domination number} of   $G$, denoted by $\gamma(G)$, is the minimum cardinality of a dominating set.  \emph{Zero forcing} is a coloring game on a graph, where the goal is to color all the vertices blue (starting with  each vertex colored blue or white). White vertices are colored blue by applying the following \emph{color change rule}:
A blue vertex $u$ can change the color of a white vertex $w$ to blue if $w$ is the unique white neighbor of $u$; in this case we say \emph{$u$ forces $w$} and write $u\to w$. 
A set $B$ is a {\it zero forcing set} of $G$ if all the vertices of $G$ can be colored blue by repeated application of the color change rule when starting with the vertices in $B$ blue and the other vertices white.
The domination step in power domination takes the set $S$ to  $N[S]$, and $S$ is a power dominating set of $G$ if and only if $N[S]$ is a zero forcing set of $G$. A blue vertex in zero forcing corresponds to an observed vertex in power domination, because  $u\in P^{[i]}(S)$ and $N_G(u)\setminus P^{[i]}(S)=\{w\}$ is equivalent to saying that after the $i$ round, $w$ is the only unobserved neighbor of $u$, so $u\to w$ is possible.

Notice that in power domination we have performed all independently possible observations simultaneously, whereas in  zero forcing as just defined, we perform one color change at a time (and choose which vertex forces $w$ if more than one vertex could force $w$).  Both perspectives are useful.  
For zero forcing, we can start with a set $B$ of blue vertices and in each round we perform  all possible forces that can be done independently of each other (this is propagation for zero forcing -- see \cite{proptime}).  
Sometimes it is necessary to record how the forcing part of the power domination process is carried out.   If $i\ge 1$ and there is at least one vertex  $u\in P^{[i]}(S)$ such that $N_G(u)\setminus P^{[i]}(S)=\{w\}$, then one such $u$ is chosen as the vertex to force $w$, denoted by $u\to w$.  In the dominating step, for each vertex $w\in N[S]\setminus S$, we choose an $x\in S$ such that $w\in N(x)$ and record $x\to w$ as a force.  When it is desired to distinguish these two kinds of forces, a force in step \eqref{domstep} is called a {\em domination force} and a  force in step \eqref{zfproc} is called a {\em zero force}. For a given set $S$, we construct the set of all observed vertices, recording each force in order. We consider only  a \emph{propagating} set of forces, in which $\rd(u)<\rd(v)$ implies $u$ is forced before $v$ in the ordered list of forces. The symbol $\F$ is used to denote the \emph{set  of forces}. %
Given a power dominating set $S$ and set of forces $\F$, a {\em forcing chain} is a sequence $v_0\to v_1\to \dots\to v_a$ such that $v_{i-1}\to v_i\in\F$ for $i=1,\dots,a$.

\begin{obs} If $v_0\to v_1\to\dots\to v_a$ is a forcing chain for a given set $\F$ of forces of a power dominating set $S$ of $G$, then $\rd(v_i)\ge i$, because $\rd(v_0)\ge 0$ and $\rd(v_{i+1})\ge \rd(v_i)+1$.
\end{obs}

\begin{rem} \label{assign} Let $S$ be a power dominating set of $G$. It is well known that the number of vertices forced  in each round of zero forcing cannot exceed the number of initial blue vertices. After the first round, power domination uses the zero forcing process, so the number of observed vertices that have an unobserved neighbor is at most  $|P^{(1)}(S)|$.  Thus $|P^{(i+1)}(S)|\leq |P^{(1)}(S)|$, for all $i\geq 0$.
\end{rem} 

Since electrical power networks are modeled by connected simple finite undirected graphs, in this work we assume every graph $G$ has these properties (although `connected' is listed as a hypothesis in  results since power domination has been studied in graphs that need not be connected). %

\section{Preliminary results}\label{s:bounds}

In this section we present bounds on the product power throttling number in terms of other graph parameters.  %
\begin{obs}\label{completethpdx}
For a connected graph $G$, $\thpdx(G)\ge 1$; moreover, $\thpdx(G)=1$ if and only if  $\gamma(G)=1$. This implies that $\thpdx(K_n)=1$ and  $\thpdx(K_{1,n-1})=1$.
\end{obs}
\begin{obs}\label{o:basic-bds-up} For every connected graph $G$, $\thpdx(G)\ge \pd(G)$ because  $|S|\ge \pd(G)$ in order to have finite propagation time and $\ptpd(G;S)\ge 1$.

\end{obs}

\begin{obs}\label{o:basic-bds} For every connected graph $G$:
\ben[$(1)$]
\item $\thpdx(G)\le \gamma(G)$, since
a minimum dominating set is a power dominating set with power propagation
time $1$.
\item $\thpdx(G)\le \pd(G)\ptpd(G)$, realized by a minimum power dominating set $S$ such that $\ppt(G;S)=\ppt(G)$.  
\een
\end{obs}

The domination number upper bound in the previous observation is explored further throughout the rest of the paper. 
Next we give an example showing that the product power   throttling number need not be the minimum of the two upper bounds $ \gamma(G)$ and $ \pd(G)\ptpd(G)$.  
The {\em spider} $S(\ell_1,\dots,\ell_k)$ has one vertex of degree $k$ and $k$ pendent paths on $\ell_1,\dots,\ell_k$ vertices, respectively.  Figure \ref{f:s722222} shows $S(7,2,2,2,2,2)$.

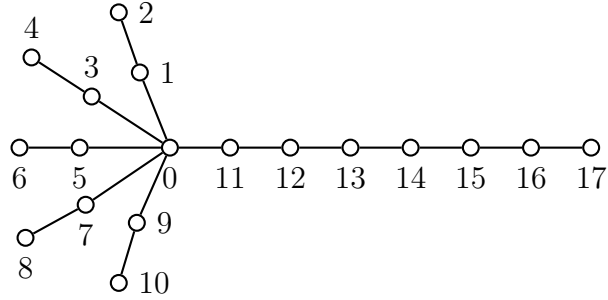
\begin{figure}[!h]
\centering
\begin{tikzpicture}[scale=.8]
\vertex (A) at (0.5,0)[label=below: $0$] {};
\vertex (B) at (1.5,0) [label=below:$11$]{};
\vertex (C) at (2.5,0) [label=below:$12$]{};
\vertex (D) at (3.5,0)[label=below:$13$]{};
\vertex (E) at (4.5,0)[label=below:$14$]{};
\vertex (F) at (5.5,0)[label=below:$15$]{};
\vertex (G) at (6.5,0)[label=below:$16$]{};
\vertex (H) at (7.5,0)[label=below:$17$] {};
\vertex (I) at (-1,0)[label=below:$5$]{};
\vertex (J) at (-2,0)[label=below:$6$]{};
\vertex (K) at (-0.8,0.85)[label=above:$3$]{};
\vertex (L) at (-1.8,1.5)[label=above:$4$]{};
\vertex (M) at (-.0,1.25)[label=right:$1$]{};
\vertex (N) at (-.35,2.25)[label=right:$2$]{};
\vertex (O) at (-0.9,-0.95)[label=below:$7$]{};
\vertex (P) at (-1.9,-1.5)[label=below:$8$]{};
\vertex (Q) at (-0.05,-1.25)[label=right:$9$]{};
\vertex (R) at (-.35,-2.25)[label=right:$10$]{};
\draw[thick] (A) to (B);
\draw[thick] (B) to (C);
\draw[thick] (C) to (D);
\draw[thick] (D) to (E);
\draw[thick] (E) to (F);
\draw[thick] (F) to (G);
\draw[thick] (G) to (H);
\draw[thick] (A) to (I);
\draw[thick] (I) to (J);
\draw[thick] (A) to (K);
\draw[thick] (K) to (L);
\draw[thick] (A) to (M);
\draw[thick] (M) to (N);
\draw[thick] (A) to (O);
\draw[thick] (O) to (P);
\draw[thick] (A) to (Q);
\draw[thick] (Q) to (R);
\end{tikzpicture}\\

\caption{The spider $S(7,2,2,2,2,2)$.  \label{f:s722222}}
\end{figure}
\begin{ex}\label{e:spider} Let $G=S(7,2,2,2,2,2)$ with the vertices numbered as in Figure \ref{f:s722222}. The product power   throttling number of $G$ is $4$ using the power dominating set $\{0,15\}$.  This cannot be realized by either a minimum dominating set (since $\gamma(G)=8$) or  a minimum power dominating set (since 
$\pd(G)=1$, $\ptpd(G)=7$, and $\thpdx(G,1)=7$).   %
\end{ex}

From Observation \ref{o:basic-bds}, only subsets $S\subseteq V(G)$ such that  $\pd(G)\le |S|\le \gamma(G)$ need be considered to determine $\thpdx(G)$. %

Next we turn our attention to lower bounds. The maximum degree of a graph $G$ is denoted by $\Delta(G)$.

\begin{thm}\label{PDlowerbd} {\rm \cite{FHKY17}} %
In a connected graph $G$, %
\[ \pd(G)\ge  \frac {|V(G)|} {\ptpd(G)\Delta(G)+1}.\]
 \end{thm}

The argument used to establish  Theorem \ref{PDlowerbd}  in \cite{FHKY17} consists of showing that for any power dominating set $S$ of $G$,

\beq \label{Sbd} |S| \ge  \frac {|V(G)|} {\ptpd(G;S)\Delta(G)+1}.\eeq

Notice that in the particular case when $S$ is a minimum power dominating set of minimum power propagation time, $|S|=\pd(G)$, $\ptpd(G;S)=\ptpd(G)$ and inequality \eqref{Sbd} gives the bound in Theorem \ref{PDlowerbd}. As we show next, in the study of throttling, inequality \eqref{Sbd} has additional consequences. The next result is immediate since $\frac {|V(G)|} {\ptpd(G;S)\Delta(G)+1}\ge \frac {|V(G)|} {\ptpd(G;S)(\Delta(G)+1)}$.

 \begin{cor}\label{t:lowerbd} In a connected graph $G$, %
\[ \thpdx(G) \ge \lc \frac {|V(G)|} {\Delta(G)+1}\rc.\]
 \end{cor}

\section{Conditions resulting in  $\thpdx(G)= \gamma(G)$}\label{s:pdpt-d}\label{s:thdpx=gamma} %

In this section we present conditions sufficient to ensure $\thpdx(G)=  \gamma(G)$.  The next result follows from Corollary \ref{t:lowerbd}.

 \begin{obs}\label{c:squeeze} Let $G$ be a connected graph of order $n$ with $\gamma(G)=\lc\frac n {\Delta(G)+1}\rc$.  Then $\thpdx(G)=\lc\frac n {\Delta(G)+1}\rc$.
\end{obs}
\begin{obs}\label{p:pathcycle} Since $\gamma(P_n)= \lc \frac n 3\rc$ and $\Delta(P_n)=2$, $\thpdx(P_n)=\lc \frac n 3\rc$. Similarly, since $\gamma(C_n)= \lc \frac n 3\rc$ and $\Delta(C_n)=2$, $\thpdx(C_n)=\lc \frac n 3\rc$.
\end{obs}

A $d$-{\em star cover}  of a graph  $G$ is a set of subgraphs $G_i=K_{1,p_i}, i=1,\dots,d$ such that $\cup_{i=1}^dV(G_i)=V(G)$. A star cover is {\em disjoint} if the vertex sets of the stars are disjoint.   For any graph $G$, any dominating set gives a star cover (which can be chosen disjoint), and  $\gamma(G)$ is the minimum $d$ such that $G$ has a $d$-star cover.  

\begin{obs}\label{p:d-tight}{\rm \cite[p. 50]{dom-book}}  A graph $G$ of order $n$ has $\gamma(G)=\frac n {\Delta(G)+1}$ if and only if $G$ has a $\frac n {\Delta(G)+1}$-star cover that is disjoint and in which each star has order  $\Delta(G)+1$.  %
\end{obs}

\begin{obs}\label{c:d-tight} If a connected graph $G$ has a star cover consisting of $d$ disjoint copies of $K_{1,d}$ and $\Delta(G)=d$, then $\thpdx(G)=d=\gamma(G)$.  %
\end{obs}

One can construct a graph $G$ of order $d(d+1)$ with $\Delta(G)=d=\gamma(G)=\frac {d(d+1)}{d+1}$ as described in the next example.

\begin{ex}\label{e:dstar} Define the graph $G_d$ to be the graph obtained from $d$ disjoint copies of $K_{1,d}$ by adding all necessary edges so that each leaf of a $K_{1,d}$ is adjacent to the corresponding leaves of the other $d-1$ copies of $K_{1,d}$.  Then $G_d$ is a $d$-regular graph of order $d(d+1)$ with $\gamma(G)=d$; $G_3$ is shown in Figure \ref{f:K13boxK3}.  
\end{ex}

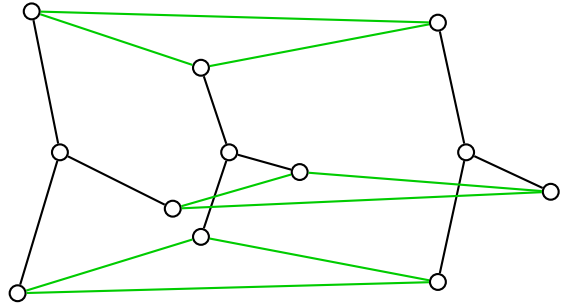
\begin{figure}[!h]
\centering
\begin{tikzpicture}[scale=.75,rotate=90]
\vertex (A) at (0,0){};
\vertex (B) at (-.35,-1.25) {};
\vertex (C) at (1.5,.5){};
\vertex (D) at (-1.5,.5){};
\vertex (E) at (2.5,3.5){};
\vertex (F) at (2.3,-3.7){};
\vertex (G) at (-2.5,3.75) {};
\vertex (H) at (-2.3,-3.7){};
\vertex (I) at (0.0,-4.2){};
\vertex (J) at (-.7,-5.7){};
\vertex (K) at (0,3){};
\vertex (L) at (-1,1){};

\draw[thick,] (A) to (B);
\draw[thick] (A) to (C);
\draw[thick] (A) to (D);
\draw[thick,green] (C) to (E);
\draw[thick,green] (C) to (F);
\draw[thick,green] (D) to (G);
\draw[thick,green] (D) to (H);
\draw[thick,green] (E) to (F);
\draw[thick,green] (G) to (H);
\draw[thick] (I) to (F);
\draw[thick] (I) to (H);
\draw[thick] (I) to (J);
\draw[thick,green] (B) to (J);
\draw[thick,green] (J) to (L);
\draw[thick,green] (B) to (L);
\draw[thick] (G) to (K);
\draw[thick] (K) to (E);
\draw[thick] (K) to (L);
\end{tikzpicture}\\

\caption{The graph $G_3$, with the edges added to $3K_{1,3}$  shown in {\gre \bf green}. \label{f:K13boxK3}}%
\end{figure}

The construction in Example \ref{e:dstar} can be relaxed by adding edges between the degree one vertices of the stars such that  no such vertex is incident with more than $d-1$ additional edges and at least $d-1$ additional edges are added to connect the graph. 
When $\frac n {\Delta(G)+1}$ is not an integer and $\gamma(G)=\lc\frac n {\Delta(G)+1}\rc$, by Observation \ref{c:squeeze} $\thpdx(G)=\gamma(G)$.  Although $G$ will have a $\gamma(G)$-star cover, it is possible that none of the stars will have
 order $\Delta(G) + 1$, as  Example \ref{e:d-tight} shows.
\begin{ex}\label{e:d-tight} Let $H$ be the graph shown in Figure \ref{f:weird}.  Then $\Delta(H)=4$,  $\{x,y,z\}$ is the unique minimum dominating set and $H$ has only one 3-star cover, in which each star is $K_{1,3}=K_{1,\Delta(H)-1}$.
\end{ex}

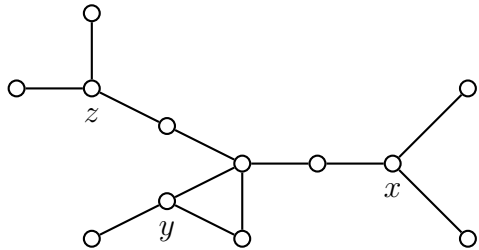
\begin{figure}[h!]
\centering
\begin{tikzpicture}[scale=1]
\vertex (A) at (0,0){};
\vertex (B) at (1,0) {};
\vertex (C) at (2,0)[label=below:$x$]{};
\vertex (D) at (-1,-0.5)[label=below:$y$]{};
\vertex (E) at (-2,-1){};
\vertex (F) at (0,-1) {};
\vertex (G) at (3,1) {};
\vertex (H) at (3,-1){};
\vertex (I) at (-1,0.5){};
\vertex (J) at (-2,1)[label=below:$z$]{};
\vertex (K) at (-2,2){};
\vertex (L) at (-3,1){};
\draw[thick] (A) to (B);
\draw[thick] (B) to (C);
\draw[thick] (A) to (D);
\draw[thick] (D) to (E);
\draw[thick] (C) to (G);
\draw[thick] (C) to (H);
\draw[thick] (A) to (F);
\draw[thick] (F) to (D);
\draw[thick] (A) to (I);
\draw[thick] (I) to (J);
\draw[thick] (J) to (K);
\draw[thick] (J) to (L);
\end{tikzpicture}\\
\caption{The graph $H$ in Example \ref{e:d-tight}.  \label{f:weird}}
\end{figure}

\begin{prop}\label{basic} In any graph $G$, $\pd(G)=\gamma(G)$ if and only if $\ptpd(G)=1$. Moreover, in this case  $\thpdx(G)=\gamma(G)$.\end{prop}

\bpf Suppose $\pd(G)=\gamma(G)$.  Then  $\ptpd(G)=1$ because any minimum dominating set $S$ of $G$ is also a minimum power dominating set of $G$ with $\ptpd(G;S) =1$.   Conversely, if  $\ptpd(G)=1$, then there exists a  minimum power dominating set $S$ of $G$ such that $\ptpd(G;S)=1$, which implies $S$ is a dominating set and thus $\pd(G)\le \gamma (G)\le |S|=\pd (G)$. The last statement follows from Observations \ref{o:basic-bds-up} and \ref{o:basic-bds}.
\epf

\begin{prop}\label{r:comp-bd-thpdx} Let $G$ be a connected graph.  %
\ben[(1)]
\item\label{r:comp-bd-thpdx-1}  Suppose $\gamma (G)\le b$. If $\pd(G)\geq {{b}\over 2}$, then $\thpdx(G)=\gamma(G)$. In particular, if $\pd(G)\ge \frac{\gamma(G)}2$, then $\thpdx(G)=\gamma(G)$.
\item\label{r:comp-bd-thpdx-2} Suppose $\thpdx(G;S)=  b< \gamma(G)$ for some $S\subset V(G)$.
Then $\thpdx(G)=\thpdx(G,k)$ for some $k$ such that  $\pd(G)\le k\le \lf\frac{b}2\rf$.
\een
\end{prop}
\bpf
Let $S$ be an arbitrary power dominating set of $G$. Then, $|S|\geq \pd(G)$ and $\thpdx(G;S)=|S|\ptpd(G;S)\geq \pd(G)\ptpd(G;S)$. If $S$ is not a dominating set, then $\ptpd(G;S)\geq 2$ and this implies $\thpdx(G;S)\geq 2|S|\ge 2\pd(G)$. 

For \eqref{r:comp-bd-thpdx-1}, if $S$ is  not a  dominating set, then  $\thpdx(G;S)\geq 2\pd(G)\geq b\ge \gamma(G)$ by hypothesis. Therefore, %
$\thpdx(G)=\gamma(G)$.  

For \eqref{r:comp-bd-thpdx-2}, $\thpdx(G,|S|)\le b$ and $|S|\le \frac b 2$ since $b<\gamma(G)$ implies $\ppt(G;S)\ge 2$.   For any graph, $\thpdx(G)=\thpdx(G,k)$ for some $k$ with $\pd(G)\le k\le \gamma(G)$. If $\lf\frac{b}2\rf<k< \gamma(G)$, then $\ppt(G,k)\ge 2$, so $\thpdx(G,k)\ge 2k> b$.    
\epf

Now we apply the first part of the previous result  together with some known upper bounds for the power domination number of a graph in terms of its order and minimum degree to  derive additional sufficient conditions for $\thpdx(G)=\gamma(G)$.

\begin{thm}{\rm \cite[Theorem 2.1]{dom-book}}\label{half}
Let $G$ be a graph of order $n$ having no isolated vertices. Then $\gamma (G) \leq \lf\frac n 2\rf$.
\end{thm}

Notice that the condition of $G$ not having isolated vertices can also be stated as $\delta (G)\geq 1$, which is immediate for a connected graph of order at least two.

\begin{cor} \label{cor-n/2} Let $G$ be a connected graph of order $n\ge 2$. If $\pd (G)\geq {n\over 4}$, then  $\thpdx(G)=\gamma(G)$.
\end{cor}

\begin{thm}{\rm \cite[Theorem 2.3]{dom-book}}\label{2:5}
Let $G$ be a connected graph of order $n$ with minimum degree $\delta (G)\geq 2$. If $G\not\in {\mathcal{A}}$ where ${\mathcal{A}}$ is the set of graphs in Figure \ref{PDSfamilyA}, then $\gamma (G) \leq {{2n}\over 5}$.
\end{thm}
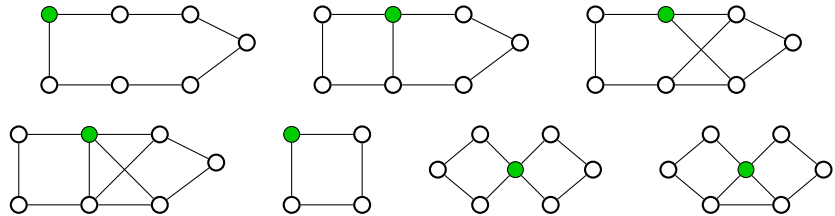
\begin{figure}[h!]
\[
\begin{array}{ccc}
		\begin{tikzpicture}[scale=.75]
	
		\vertex (A) at (0,0) {};
		\vertex (B) at (1.25, 0) {};
		\vertex (C) at (2.5, 0) {};
		\vertex (D) at (3.5,0.75) {};
		\vertex (E) at (2.5, 1.25) {};
		\vertex (F) at (1.25, 1.25) {};
		\gvertex (G) at (0,1.25) {};
		\draw (A) to (B); 
		\draw (B) to (C); 
		\draw (C) to (D); 
		\draw (D) to (E); 
		\draw (E) to (F); 
		\draw (F) to (G); 
		\draw (G) to (A); 
	\end{tikzpicture}
	\hskip 1em

&

\begin{tikzpicture}[scale=.75]
		\vertex (A) at (0,0) {};
		\vertex (B) at (1.25, 0) {};
		\vertex (C) at (2.5, 0) {};
		\vertex (D) at (3.5,0.75) {};
		\vertex (E) at (2.5, 1.25) {};
		\gvertex (F) at (1.25, 1.25) {};
		\vertex (G) at (0,1.25) {};
		\draw (A) to (B); 
		\draw (B) to (C); 
		\draw (C) to (D); 
		\draw (D) to (E); 
		\draw (E) to (F); 
		\draw (F) to (G); 
		\draw (G) to (A); 
		\draw (F) to (B);
	\end{tikzpicture}
	\hskip 1em
&
\begin{tikzpicture}[scale=.75]
		\vertex (A) at (0,0) {};
		\vertex (B) at (1.25, 0) {};
		\vertex (C) at (2.5, 0) {};
		\vertex (D) at (3.5,0.75) {};
		\vertex (E) at (2.5, 1.25) {};
		\gvertex (F) at (1.25, 1.25) {};
		\vertex (G) at (0,1.25) {};
		\draw (A) to (B); 
		\draw (B) to (C); 
		\draw (C) to (D); 
		\draw (D) to (E); 
		\draw (E) to (F); 
		\draw (F) to (G); 
		\draw (G) to (A); 
		\draw (C) to (F);
		\draw (B) to (E);
	\end{tikzpicture}
\end{array}
\]
\[
\begin{array}{cccc}
\begin{tikzpicture}[scale=.75]
		\vertex (A) at (0,0) {};
		\vertex (B) at (1.25, 0) {};
		\vertex (C) at (2.5, 0) {};
		\vertex (D) at (3.5,0.75) {};
		\vertex (E) at (2.5, 1.25) {};
		\gvertex (F) at (1.25, 1.25) {};
		\vertex (G) at (0,1.25) {};
		\draw (A) to (B); 
		\draw (B) to (C); 
		\draw (C) to (D); 
		\draw (D) to (E); 
		\draw (E) to (F); 
		\draw (F) to (G); 
		\draw (G) to (A); 
		\draw (B) to (F);
		\draw (C) to (F);
		\draw (B) to (E);
	\end{tikzpicture}
	\hskip 1em
&

\begin{tikzpicture}[scale=.75]
		\vertex (A) at (0,0){};
		\vertex (B) at (1.25, 0){};
		\vertex (C) at (1.25, 1.25){};
		\gvertex (D) at (0, 1.25){};
		\draw (A) to (B); 
		\draw (B) to (C); 
		\draw (C) to (D); 
		\draw(D) to (A);

	\end{tikzpicture}
\hskip 1em
&
	\begin{tikzpicture}[scale=.75]
		\vertex (A) at (0,0) {};
		\vertex (B) at (1.25, 0) {};
		\vertex (C) at (2, 0.625) {};
		\vertex (D) at (1.25,1.25) {};
		\vertex (E) at (0, 1.25) {};
		\vertex (F) at (-.75, .625) {};
		\gvertex (G) at (0.625,.625) {};
		\draw (B) to (C); 
		\draw (B) to (G); 
		\draw (C) to (D); 
		\draw (D) to (G); 
		\draw (E) to (G); 
		\draw (F) to (E); 
		\draw (G) to (A); 
		\draw (F) to (A);
	\end{tikzpicture}
	\hskip 1em
	&	
	\begin{tikzpicture}[scale=.75]
		\vertex (A) at (0,0) {};
		\vertex (B) at (1.25, 0) {};
		\vertex (C) at (2, 0.625) {};
		\vertex (D) at (1.25,1.25) {};
		\vertex (E) at (0, 1.25) {};
		\vertex (F) at (-.75, .625) {};
		\gvertex (G) at (0.625,.625) {};
		\draw (A) to (B); 
		\draw (B) to (C); 
		\draw (B) to (G); 
		\draw (C) to (D); 
		\draw (D) to (G); 
		\draw (E) to (G); 
		\draw (F) to (E); 
		\draw (G) to (A); 
		\draw (F) to (A);
	\end{tikzpicture}\end{array}
\]

\caption{The graphs in the family ${\mathcal{A}}$. Each graph has a vertex (colored {\gre \bf green}) that is a minimum power dominating set.}\label{PDSfamilyA}
\end{figure}

\begin{cor} \label{cor_2:5} Let $G$ be a connected graph of order $n$ with minimum degree $\delta (G)\geq 2$. If %
$\pd (G)\geq {n\over 5}$, then $\thpdx(G)=\gamma(G)$.
\end{cor}

\bpf
The graph $C_4$ has  $\thpdx(C_4)=2=\gamma(C_4)$ and thus satisfies the conclusion. %
The remaining graphs $H\in {\mathcal{A}}$ do not satisfy the hypothesis %
since each has order $n=7$ and $\pd(H)=1 <\frac n 5$. %
When $\pd (G)\geq \frac n 5$ (and $G\neq C_4$) we have $2\pd(G) \ge \frac{2n}5 \ge \gamma(G)$ by Theorem \ref{2:5} and then the conclusion follows from Proposition \ref{r:comp-bd-thpdx}\eqref{r:comp-bd-thpdx-1}.
\epf

\begin{thm} {\rm \cite[Theorem 2.7]{dom-book}} \label{3:8}
Let $G$ be a graph of order $n$ with minimum degree $\delta (G)\geq 3$. Then, $\gamma (G) \leq {{3n}\over 8}$.
\end{thm}

\begin{cor} \label{cor_3:8}Let $G$ be a connected graph of order $n$  with minimum degree $\delta (G)\geq 3$.  If $\pd (G)\geq {{3n}\over {16}}$, then  $\thpdx(G)=\gamma(G)$.
\end{cor}

It should be noted that while the conditions above are sufficient, they are not necessary.  Graphs with very low power domination number may still realize $\thpdx(G)=\gamma(G)$.  For example, $\pd(P_n)=1$ and  $\thpdx(P_n)=\gamma(P_n)=\lc\frac n 3\rc$  (Observation \ref{p:pathcycle}).  

There are also many other families of graphs that have $\thpdx(G)=\gamma(G)$, including unit interval graphs (described in Section \ref{s:unit-interval}) and some Cartesian products (described in Section \ref{s:Cart-prod}).

\section{Extreme product power   throttling numbers}\label{s:extreme}

In this section we characterize graphs with extremely low and high product power throttling numbers.%

\subsection{Low product power   throttling numbers}

By Observation \ref{completethpdx}, any graph $G$ has $\thpdx(G)\ge1$ and  $\thpdx(G)=1$ if and only if $\gamma(G) = 1$. The following result characterizes  graphs $G$ for which $\thpdx(G)=2$.  %

\begin{thm}\label{t:thpdx2}
A connected graph $G$ has $\thpdx(G)=2$ if and only if $G$ satisfies one or both of the following conditions:
\ben[(a)] 
\item\label{i:d=2} $\gamma(G)=2$.  
\item\label{i:pt=2} $\pd(G)=1$ and $\ptpd(G)=2$.  
\een
\end{thm}

\bpf
Suppose $G$ is a graph satisfying at least one of the conditions. Then  $\thpdx(G)\le 2$ by Observation \ref{o:basic-bds}.  If    $\gamma(G)=2$ or $\ptpd(G)=2$, then  %
$\thpdx(G)\ge 2$.%

Conversely, assume $G$ is a graph with $\thpdx(G)=2$, and let $S\subset V(G)$ such that $\thpdx(G;S)=\thpdx(G)$. %
There are only two possibilities:  $|S|=1$ and $\ptpd(G;S)=2$, or $|S|=2$ and $\ptpd(G;S)=1$. 
Suppose first that $|S|=2$ and $\ptpd(G;S)=1$.  Then $S$ is a dominating set of $G$ because  $\ptpd(G;S)=1$, and we conclude $\gamma (G)\leq |S|=2$. Furthermore, $\thpdx(G)=2$ implies $\gamma(G)\ge 2$, so $\gamma(G)= 2$. %
Finally, consider  the case $|S|=1$ and $\ptpd(G;S)=2$. Since $|S|=1$,   $S$ must be a minimum power dominating set of $G$ and  $\pd(G)=1$. Furthermore, $\thpdx(G)=2$ implies $\ptpd(G)\ge 2$. %
\epf

\begin{rem} Any connected graph satisfying Theorem \ref{t:thpdx2}\eqref{i:pt=2} can be constructed as follows: Start with any graph $H$ of order at least two such that $\gamma(H)=1$. %
Suppose vertex $u$ is adjacent to every other vertex and let the remaining vertices of $H$ be denoted by $v_1,\dots,v_k$.   %
Add  $1\le \ell\le k$ additional vertices $\{w_1,\dots,w_\ell\}$ with $v_i$ adjacent to $w_i$ and to none of the other $w_j$.  Add any subset (possibly empty) of the edges $\{ v_sw_j:s=\ell+1,\dots,k, j=1,\dots, \ell\}$ and  any subset (possibly empty) of edges of the form $w_iw_j$.
\end{rem}

Observe that conditions \eqref{i:d=2} and \eqref{i:pt=2} can hold simultaneously. For example, if $G=C_5$ or $G=P_5$, then   $\pd(G)=1$ and $\ptpd(G)=2$, and $\gamma(G)=2$. %

\subsection{High product power   throttling numbers}

We know $\thpdx(G)\le\gamma(G)\le \frac n 2$ for any connected graph $G$ of order $n\ge 2$ %
(Theorem \ref{half}).  In this section we characterize graphs having $\thpdx(G)= \frac n 2$.  %

\begin{rem}\label{r:no-leaf} It is known that if a connected graph has a high-degree ($\ge 3$) vertex, then there is a minimum power dominating set in which each vertex has degree at least three. It is not true that for every graph with a high-degree vertex there is an optimal set for product power throttling that has all high-degree vertices.  For example, the spider  $S(4,1,1)$ has $\thpdx(S(4,1,1))=\gamma(S(4,1,1))=2$  but the power propagation time of the only high-degree vertex is 4.
 It is  true that for any graph that has at least one vertex of degree two or more, there is an optimal set for product power throttling in which all vertices have degree at least two (no leaves).  This can be seen by replacing each leaf in an optimal set for product power   throttling by its neighbor (no redundancies can be created or the set would not have been optimal).
\end{rem} 

For a graph $H$, the {\em corona} of $H$ with $K_1$,  denoted by $H\circ K_1$, is the graph obtained from $H$ by appending a leaf to each vertex of $H$. 

\begin{thm}\label{half-ratio}
If $H$ is a connected graph of order at least two and $G=H\circ K_1$, then $\thpdx(G) = 2\gamma(H)$.  Furthermore, any power dominating set for $G$ that is a subset of $V(H)$ 
must be a dominating set for $H$.
\end{thm}

\begin{proof}
First we show $\thpdx(G) \le 2\gamma(H)$.    Let $S$ be a dominating set of $H$ with $|S| = \gamma(H)$. After the first round,  all vertices of $H$ are observed, each vertex of $H$ has at most one unobserved neighbor, and each unobserved vertex has an observed neighbor.  Thus, all vertices of $G$ are observed after the second round.  %

  Next we prove that any power dominating set for $G$ that is a subset of $V(H)$  must be a dominating set for $H$.  Let  $S\subseteq V(H)$ be a power dominating set of $G$.  %
 Suppose  that there exists a vertex $w \in V(H)$ that remains unobserved after the first round, and thus none of $w$'s neighbors are in $S$.  Every $u \in V(H)$ adjacent to $w$ is also adjacent to at least one additional unobserved vertex (its leaf neighbor).  Thus $w$ will never be observed by one of its neighbors and this contradicts the assumption that $S$ is a power dominating set.
 
Finally, we show $\thpdx(G) \ge 2\gamma(H)$.  First consider the case that $\thpdx(G)$ is realized by a power dominating set $S$ with power propagation time at least two. Without loss of generality, we may assume $S \subseteq V(H)$ (cf. Remark \ref{r:no-leaf}).   
Then $S \ge \gamma(H)$  since we proved above that  $S$ is a dominating set of $H$.  Thus $\thpdx(G) = |S|\ptpd(G;S) \ge 2\gamma(H)$.
Now consider the case in which $\thpdx(G)$ is realized by a dominating set $S$ of $G$, so $\thpdx(G) = \gamma(G)$.    Observe that $\gamma(G) = |V(H)|$ since  $G$ has $|V(H)|$ leaves and each leaf must be dominated by a different vertex of $G$. %
Thus  $\thpdx(G) = \gamma(G) = |V(H)| \ge 2 \gamma(H)$ by Theorem \ref{half}. %
\end{proof}

Since $\gamma(H'\circ K_1)=n'$ for any   connected graph $H'$ of order $n'$, the next result is immediate.

\begin{cor}\label{c:half}
If $H$ is a connected graph of order $n$ and $G=(H \circ K_1) \circ K_1$, then $\thpdx(G)=2n=\frac 1 2 |V(G)|$.  \end{cor}

We use the next characterization of graphs $G$ of order $n$ having $\gamma(G)=\frac n 2$ to characterize graphs having $\thpdx(G)=\frac n 2$.

\begin{thm}\label{t:dom-half}{\rm \cite[Theorem 2.2]{dom-book}}
A  connected graph  $G$  of order $n\ge 2$ has $\gamma(G)=\frac n 2$ if and only if $G=G'\circ K_1$ for some connected graph $G'$ or  $G=C_4$.  \end{thm}

\begin{thm}\label{half-ratio-iff}
A  connected graph  $G$  of order $n\ge 2$ has $\thpdx(G)=\frac n 2$ if and only if $G=(H\circ K_1)\circ K_1$ for some connected graph $H$,  $G=C_4\circ K_1$, or $G=C_4$.   \end{thm}
\bpf Corollary \ref{c:half} implies $\thpdx((H\circ K_1)\circ K_1)=\frac 1 2 |V((H\circ K_1)\circ K_1)|$ and Theorem \ref{half-ratio} implies $\thpdx(C_4\circ K_1)=4$; it is easy to see that  $\thpdx(C_4)=2$.

Assume $G$ is a connected graph of order $n$ such that $\thpdx(G)=\frac n 2$.  Then by Observation \ref{o:basic-bds} and Theorem \ref{half}, $\gamma(G)=\frac n 2$.  Then $G=G'\circ K_1$ for some connected graph $G'$ or $G=C_4$ by Theorem \ref{t:dom-half}.  If $n=2$, then $G=K_1\circ K_1$, so assume $n\ge 4$.  Assume $G=G'\circ K_1$, so    $\frac n 2=\thpdx(G)=2\gamma(G')$ by Theorem \ref{half-ratio}. Thus $\gamma(G')=\frac n 4 = \frac{|V(G')|} 2$, so  $G'=H\circ K_1$ for some connected $H$ or $G'=C_4$.  %
\epf

\begin{figure}[h!]
\centering
\begin{tikzpicture}[scale=.75]
    \draw[thick] (0,0) ellipse (2cm and 1cm);
    \vtx(A) at (0,1)[label=below:$u$] {};
    \vtx(B) at (.5,1.75)[label=right:$z_u$]{};
    \vtx(C) at (-.5,1.75)[label=left:$y_u$]{};
    \vtx(D) at (-.5,2.75)[label=left:$x_u$]{};
    
    \draw[thick] (A) to (B);
    \draw[thick] (A) to (C);
    \draw[thick](C) to (D);
    
\end{tikzpicture}
\caption{Constructing a graph with product power   throttling number equal to half its order.  \label{f:Gcirc}}
\end{figure}
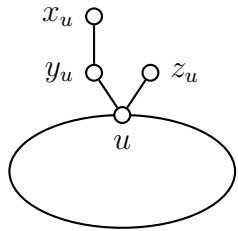
A graph $G=(H \circ K_1) \circ K_1$ can also be constructed from a connected graph $H$ of order at least one by  appending to each vertex $u$ of $H$ a path of length two (with vertices $x_u$ and $y_u$) and a path of length one (with vertex $z_u$) as shown in Figure~\ref{f:Gcirc}.  Observe that the order of a graph $G=(H \circ K_1) \circ K_1$ is always divisible by four.

\section{Unit interval graphs}\label{s:unit-interval}

A graph $G = (V,E)$ is a \emph{unit interval graph} if each vertex $v \in V$ can be assigned a closed unit length real interval $I(v)$ so that vertices are adjacent precisely when their assigned intervals intersect.  In symbols, for $x, y \in V(G)$ we have $xy \in E(G)$ if and only if $I(x) \cap I(y) \neq \emptyset$.    Any unit interval graph has a unit interval representation in which all the interval endpoints are distinct, and we assume all our representations have this property.
  See \cite{GoTr04} for additional background.  
It is convenient to write $I(v) = [\ell(v),r(v)]$ where $r(v) - \ell(v) = 1$.  If $G$ is a unit interval graph with a fixed representation, we refer to $\ell(v)$ as the {\em left endpoint} of the vertex $v$ (as well as the left endpoint of the interval $I(v)$), and analogously for $r(v)$.

In Theorem \ref{unit-int-gr-thm}, we show that the  product power throttling number of a unit interval graph is its domination number.  The proof of Theorem~\ref{unit-int-gr-thm} will depend on several lemmas.
 We begin with some additional notation. %
   Let $G$ be a unit interval graph and fix a unit interval representation $I(v) = [\ell(v),r(v)]$. Then the order of the left endpoints provides an order on the vertices, called the {\em induced order}.  That is, $v<u$ if and only if $\ell(v)<\ell(u)$.

 \begin{obs}\label{o:UI-consec}  Let $G$ be a connected unit interval graph with a fixed unit representation.   For each vertex $v$, the closed neighborhood $N[v]$ is a consecutive  set of vertices in the induced order. \end{obs}

  The next lemma shows that in a unit interval graph with a fixed representation, the order of the vertices in a forcing chain $v_0\to v_1\to\dots \to v_i$ either follows the induced order $<$ on the vertices, or follows the reverse order (i.e., $v_{j-1}>v_j$ for $j=1,\dots,i$).  Recall that $\rd(v)$ is the number of the round in which vertex $v$ is first observed.

 \begin{lem}\label{consec-nhbd-direct-force} Let $G$ be a connected unit interval graph with a fixed unit representation and induced order of the vertices, let $S$  be a power dominating set of $G$, let $\F$ be a set of forces,  let $v_0\to v_1\to\dots \to v_i$ be a forcing chain, and let   $\rd(v_i)=k\ge i$. If $v_0<v_1$, then $v_1 < v_2 < \cdots < v_i$. %
Moreover,  $\rd(u) \le k-1$ for all $u$   such that  $v_0\le u< v_i$ when $i\ge 2$.  %
  Analogous statements are true when  $v_0 > v_1$.
   \end{lem}
\bpf
Both statements are proved together by induction assuming $v_0<v_1$. If $i=1$, then there is nothing to prove.  %
Now suppose $i\ge 2$ (so $\rd(v_i)\ge 2$) and the statement is true for $i-1$. %
That is, $v_1< v_2<\dots<v_{i-1}$ and  $u\in P^{[k-1]}(S)$ for all $u$ such that $v_0\le u\le v_{i-1}$.  Suppose to the contrary that $v_i < v_{i-1}$.  If $v_0\le v_i$,  then $v_0\le v_i<v_{i-1}$ implies $v_i\in P^{[k-1]}(S)$, contradicting $\rd(v_i)=k$. 
Suppose $v_i<v_0$. Then $v_i\in N(v_{i-1})$ and $v_i<v_0<v_{i-1}$ imply $v_i\in N(v_0)$  by Observation \ref{o:UI-consec}.
Since $v_0\to v_1$ in round $k'\le k-1$,  $v_i\in P^{[k-2]}(S)$, contradicting $\rd(v_i)=k$.  Thus $v_i > v_{i-1}$.  Since $v_{i-1}\to v_i$ in round $k$, 
every other neighbor of $v_{i-1}$ is in $P^{[k-1]}(S)$, i.e., $ v_{i-1}\le u< v_i$ implies $u\in P^{[k-1]}(S)$.  
\epf

\begin{lem}
 \label{unit-int-gr-lem-1}
Let $G$ be a connected unit interval graph with initial power dominating set $S$. Then $|P^{(k)}(S)| \le 2|S|$ for every $k \ge 2$.
 \end{lem}
 
 \begin{proof}  Fix a unit interval representation of $G$ with induced order $<$. %
Let $S= \{s_1,s_2, \dots,$ $ s_p\}$ where $s_1 < s_2 < \dots < s_p$.     %
Suppose that $\rd(y)=k$ for some $k \ge 2$ and  $s_i < y < s_{i+1}$ for some $i$ with $1 \le i \le p-1$.   
There exists a forcing chain  $y_0\to y_1\to \dots\to y_{a-1}\to y_a =y$ with $a\le k$ and $y_0\in S$.  If $y_0<y$, then $\rd(x)\le k-1$ for all $x$ such that $y_0\le x <y$  by Lemma \ref{consec-nhbd-direct-force}, and if $y_0>y$, then then $\rd(x)\le k-1$ for all $x$ such that $y_0\ge x >y$.   Since $y_0\le s_i<y$ (or $y<s_{i+1}\le y$) there are at most two vertices in $P^{(k)}(S)$  between $s_i$ and $s_{i+1}$.  Similarly, there is at most one vertex  in $P^{(k)}(S)$  before $s_1$ and at most one vertex  in $P^{(k)}(S)$  after $s_p$.  Thus $|P^{(k)}(S)| \le 2|S|$ for every $k \ge 2$. \end{proof} 
 
 In power domination (and zero forcing), when a vertex $x$ is first observed in round $k$ (i.e., $\rd(x)=k$), it is not always the case that $x$ has a neighbor $y$ with $\rd(y) = k-1$. %
However, the next lemma shows that this must happen in a unit interval graph. %

 \begin{lem}
 \label{unit-int-gr-lem-2}
 If $G$ is a connected unit interval graph with power dominating set $S$,  then for each $k \ge 1$, every vertex in $P^{(k)}(S)$ is adjacent to a vertex in $P^{(k-1)}(S)$.
 \end{lem}

\begin{proof}
Fix a unit interval representation of $G$ with induced order  $<$, and let $S$ be a power dominating set. The result is clearly true (for any graph) for $k=1$ and $k=2$, so we assume $k \ge 3$,  and let $\rd(x)=k$.  Assume to the contrary that $\rd(y)\ne k-1$ for all $y\in N(x)$.  Let $z$ be a neighbor of $x$ such that $z\to x$, so $ 1 \le  \rd(z) \le k-2$.  Since $z$ did not observe $x$ in round $k-1$, $z$ must have another neighbor $w$ such that $\rd(w) = k-1$.    Since $x$ has no neighbors in $P^{(k-1)}(S)$, we know $xw\not \in E(G)$.  Thus $z$ is adjacent to both $x$ and $w$, which are not adjacent to each other.   Thus $I(z)$ intersects both $I(x)$ and $I(w)$, but $I(x) \cap I(w) = \emptyset$.    Therefore, either $w<z<x$ or $x<z<w$.  In either case, any vertex $v$  such that  $v\to z$  in round $\rd(z)\le k-2$ would also have been adjacent to a second unobserved vertex ($x$ or $w$) at that time.  This is a contradiction to the rules of power domination  if $k \ge 4$.    If $k =3$ then $v \in S$ and consequently  either $x$ or $w$   is in $P^{(1)}(S)$, 
a contradiction because $x \in P^{(k)}(S) = P^{(3)}(S)$ and $w \in P^{(k-1)}(S) = P^{(2)}(S)$.
\end{proof}

For a unit interval graph $G$, fix a unit representation of  $G$ with induced order $<$ and a power dominating set $S = \{s_1,s_2,s_3,\ldots, s_p\}$ where $s_1 < s_2 < \dots < s_p$.    Let $u_i$ be the least neighbor of $s_i$ in the order (if such a neighbor exists) and similarly, let $v_i$ be the greatest neighbor of $s_i$  (if such a neighbor exists).  
This is illustrated in Figure~\ref{int-fig}.    Let $T(S) = \{u_1,u_2, \ldots, u_p\} \cup \{v_1,v_2,  \ldots , v_p\}$.    By construction, each vertex of $S$ contributes at most 2 vertices to $T(S)$, so $|T(S)| \le 2|S|$.    The next lemma shows that $T(S)$ dominates the vertices in $S \cup P^{(1)}(S) \cup P^{(2)}(S)$.

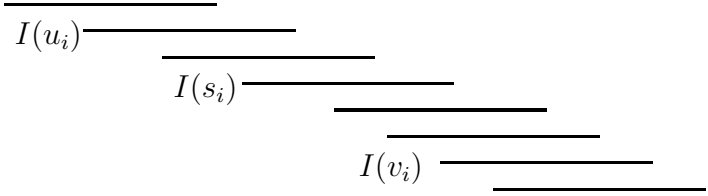
\begin{figure}
\begin{center}
 
\begin{picture}(400,80)
\thicklines

 \multiput(0,80)(30,-10){4}{\line(1,0){80}}
\put(4, 65){$I(u_i)$}
\put(64, 45){$I(s_i)$}

 \multiput(125,40)(20,-10){4}{\line(1,0){80}}
\put(134, 15){$I(v_i)$}
\end{picture}
\end{center}
\caption{Illustrating the choice of $u_i$ and $v_i$ for a given $s_i$.}
\label{int-fig}
\end{figure}

 \begin{lem}
 \label{unit-int-gr-lem-3}
 Let $G$ be a connected unit interval graph of order at least two with a fixed unit representation and induced order, and let   $S$ be  a  power dominating set of $G$.  If $T(S)$ is the subset of $ P^{(1)}(S)$ defined above, then every vertex in $S \cup  P^{(1)}(S) \cup  P^{(2)}(S)$ is  dominated by a vertex in $T(S)$.
 \end{lem}

\begin{proof}
Since we are considering only  graphs that are connected and nontrivial,  $T(S)$ dominates $S$ by construction.  Next we show $T(S)$ dominates $P^{(1)}(S)$.
By definition, any vertex $z \in P^{(1)}(S)$ is adjacent to some $s_i \in S$, so for that value of $i$ we have $I(z) \cap I(s_i) \neq \emptyset$.  If $I(z) $ contains the left endpoint $\ell(s_i)$, then $z=u_i$ or $zu_i \in E(G)$, so $z$ is either  an element of $T(S)$ or dominated by an element of $T(S)$.  The case in which $I(z) $ contains the right endpoint $r(s_i)$ is similar.  Thus  $T(S)$ dominates $P^{(1)}(S)$

Finally we show $T(S)$ dominates the vertices in $P^{(2)}(S)$.  Consider $w \in P^{(2)}(S)$.  Suppose  $s_i\to z\to w$ with $s_i > z >w$ and $\rd(z) = 1$.    By construction, $\ell(z) \ge \ell(u_i)$, so $I(u_i)$ also intersects $I(w)$, and thus $w$ is adjacent to a vertex in $T(S)$.  The case in which $s_i < z <w$ is similar.  This completes the proof.
\end{proof}

We are now ready to prove the main result of this section.

 \begin{thm}
 \label{unit-int-gr-thm}
 If $G$ is a connected unit interval graph, then $\thpdx(G) = \gamma(G)$.
 \end{thm}
 
 \begin{proof}  
Let $G$ be a connected unit interval graph of order at least two with a fixed unit representation and induced order. Let   $\thpdx(G) =\thpdx(G;S)= |S|t$ where  $t = \ptpd(G;S)$.  We consider three cases:
\ben[(i)]
\item\label{ui-thm=i}  $t=1$,  
\item\label{ui-thm=ii} $t$ an even integer greater than 1, and  
\item\label{ui-thm=iii} $t$ an odd integer greater than 1.  
\een 
It suffices to show $\gamma(G) \le   \thpdx(G)  $.  %

 \eqref{ui-thm=i}:  Since $t=1$, $S$ is a dominating set  and 
 $\gamma(G) \le |S| = |S|t  = \thpdx(G)  $.     
 
 Otherwise, we may assume $t \ge 2$.     Let $T(S)$ be the set defined just before Lemma~\ref{unit-int-gr-lem-3}.
 
\eqref{ui-thm=ii}: Assume $t$ is even.  Let $\hat S = T(S) \cup P^{(3)}(S) \cup P^{(5)}(S) \cup \dots \cup P^{(t-1)}(S)$.    By Lemma~\ref{unit-int-gr-lem-3}, $T(S)$ dominates $S \cup P^{(1)}(S)\cup P^{(2)}(S)$, and by Lemma~\ref{unit-int-gr-lem-2},  the vertices in $P^{(2j)}(S)$  are dominated by the set $P^{(2j-1)}(S)$ for $2 \le j \le \frac t 2$.  Thus $\hat S$ is a dominating set for $G$ and  $ |\hat S| = |T(S)| + |P^{(3)}(S)|+|P^{(5)}(S)|+ \dots + | P^{(t-1)}(S)|$.    By Lemma~\ref{unit-int-gr-lem-1}, $|P^{(k)}(S)| \le 2|S|$ for every $k \ge 2$, and as we noted just before Lemma~\ref{unit-int-gr-lem-3}, $|T(S)| \le 2|S|$.  Thus  
{\small \[\gamma(G) \le   |\hat S| = |T(S)| + |P^{(3)}(S)|+|P^{(5)}(S)|+ \dots + | P^{(t-1)}(S)| \le (2|S|) \frac t 2 = |S|t = \thpdx(G).\]} 
 
\eqref{ui-thm=iii}: Assume $t$ is odd.  Let $\hat S = S \cup P^{(2)}(S) \cup P^{(4)}(S) \cup \dots \cup P^{(t-1)}(S)$.   The vertices in $P^{(1)}(S)$ are dominated by   $S$ by definition, and   the vertices in $P^{(2j+1)}(S)$ are dominated by the set $P^{(2j)}(S)$ for $1 \le j \le \frac{t-1}2$ by  Lemma~\ref{unit-int-gr-lem-2}.   
 Thus  $\hat S$ is a dominating set for $G$ and $\gamma(G)\le \thpdx(G)$ as in case  \eqref{ui-thm=ii}.
  \end{proof} 
 
 We observe that the domination number of a connected unit interval graph can be found from a unit interval representation of $G$ using the following greedy algorithm.  Let $G = (V,E)$ be a unit interval graph where  $V = \{v_1, v_2,  \dots,v_n\} $,  interval $I(v_i)$  is assigned to vertex $v_i$ for each $i$, and,  $\ell(v_1) < \ell(v_2) < \dots<\ell(v_n)$.   Start with $S = \emptyset$ and add $v_k$ to $S$ where $k$ is maximum so that $I(v_1) \cap I(v_k) \neq \emptyset$.  Now remove $v_k$ and its neighbors from $G$ and iterate.  This produces a dominating set for $G$. More generally, the dominating number of  interval graphs (and several related graph classes) can be  computed in polynomial time \cite{Farber84}.

Theorem \ref{unit-int-gr-thm} need not be true for interval graphs in general, as shown by the next example.

\begin{ex} Let $G$ be the graph shown in Figure \ref{f:interval}.  Then $\gamma(G)=3$ and $\thpdx(G)=\thpdx(G;\{3\})=1\cdot 2=2$.  Observe that $I(1)=[0,3]$, $I(2)=[2,5]$, $I(3)=[4,9]$, $I(4)=[8,11]$, $I(5)=[10,13]$, $I(6)=[6,7]$ is an interval representation of $G$.  Furthermore, $G$ is not a unit interval graph since $G[\{2,3,4,6\}]$ is a $K_{1,3}$, which is prohibited for a unit interval graph.
\end{ex}
\begin{figure}[!h]
\centering

\begin{tikzpicture}[scale=.8]
\begin{scope}[very thick, every node/.style={sloped,allow upside down}]
\vertex (A) at (0,0)[label=above right:$3$] {};
\vertex (B) at (-1.5, -0.75)[label=above:$2$] {};
\vertex (C) at (-3, -1.55)[label=above:$1$]{};
\vertex (D) at (1.5, -0.75)[label=above:$4$]{};
\vertex (E) at (3,-1.55)[label=above:$5$]{};
\vertex (F) at (0,1.75)[label=above:$6$]{};
\draw(A) to (B);
\draw(A) to (D);
\draw(A) to (F);
\draw(C) to (B);
\draw(D) to (E);
\end{scope}
\end{tikzpicture}\\
\caption{An interval graph  $G$ with $\thpdx(G)<\gamma(G)$.    \label{f:interval}}
\end{figure}

\section{Cartesian products}\label{s:Cart-prod}

The \emph{Cartesian product} $G\Box H$ of graphs $G$ and $H$ is the graph whose vertex set is $V(G\Box H) = V(G) \times V(H)$ where two vertices $(x_1,y_1)$ and $(x_2,y_2)$ are adjacent in $G \Box H$ if either $x_1=x_2$ and $y_1y_2\in E(H)$ or $y_1=y_2$ and $x_1x_2\in E(G)$. In this section we provide bounds on the product power throttling number of Cartesian products.  We show that the product power throttling number equals the domination number for some families of Cartesian products, including  grid graphs (on the plane, cylinder, and torus) and Cartesian products of complete graphs with complete graphs; we also exhibit examples of Cartesian products where the product power throttling number does not equal the domination number.

\subsection{Bounds} 

We begin with upper bounds.  We know that $\thpdx(G \Box H) \leq \gamma(G \Box H)$ and\break $\thpdx(G \Box H) \leq \gamma_P(G \Box H) \ptpd(G \Box H)$ by Observation~\ref{o:basic-bds}.  The next result uses the structure of a Cartesian product to obtain additional upper bounds.

\begin{thm}\label{Cartesian-thm}
For any connected graphs $G$ and $H$, \[\thpdx(G \Box H) \leq \thpdx(G)|V(H)|\ \mbox{ and }\ \thpdx(G \Box H) \leq  \thpdx(H)|V(G)|.\]
\end{thm}
\begin{proof}
 Choose a set $S$ such that $\thpdx(G; S) = \thpdx(G)$, which implies that $ \thpdx(G)=|S| \ptpd(G; S) $. Let $S' = S \times V(H)$; that is, $S'$ is the set of vertices associated with $S$ in each copy of $G$. Since $S'$ will power dominate $G \Box H$ using each copy of $S$ simultaneously, $S'$ is a power dominating set of $G \Box H$ and $\ptpd(G \Box H; S') \le \ptpd(G; S)$. Thus, 
\[\thpdx(G \Box H) \leq |S'| \ptpd(G \Box H; S') = |S| |V(H)| \ptpd(G; S) = \thpdx(G) |V(H)|.
\]
Similarly, $\thpdx(G \Box H) \leq \thpdx(H)|V(G)|$. 
\end{proof}

The Cartesian product  in the next example  achieves one of the bounds in Theorem~\ref{Cartesian-thm} that is less than  $\gamma(G \Box H)$ and  $\gamma_P(G \Box H) \ptpd(G \Box H)$. 

\begin{ex} 
Let $G=S(7,2,2,2,2,2)$ as shown in Figure~\ref{f:s722222}. Consider the graph $G\Box P_2$. We  show that  %
$\thpdx(G \Box P_2) = \thpdx(G)|V(P_2)|=8<\gamma(G \Box P_2)=10<\pd(G \Box P_2) \ptpd(G \Box P_2)=14$.   

 We compute $\gamma(G \Box P_2)=10$ and     $\pd(G\Box P_2)=2$  \cite{sage}.   Recall that $\thpdx(G)=4$ was established in Example \ref{e:spider} and $|V(P_2)|=2$, so $8=\thpdx(G)|V(P_2)|\ge \thpdx(G \Box P_2)$.
To show that $\thpdx(G \Box P_2)= 8$,  by Proposition \ref{r:comp-bd-thpdx}\eqref{r:comp-bd-thpdx-2}  we need consider only power dominating  sets  $S$ such that  $2\le |S|\le\frac 8 2$, and since $4<\gamma(G \Box P_2)$, we know $\thpdx(G \Box P_2,4)=8$.   We  compute $\ppt(G\Box P_2,2)=7$ \cite{sage}, which implies\break $\pd(G \Box P_2) \ptpd(G \Box P_2)=14$, and  $\ppt(G\Box P_2,3)=4$ \cite{sage}, so  $\thpdx(G \Box P_2,3)=12$.
\end{ex}

Next we construct an example of a Cartesian product  that has  power product throttling number less than  $\gamma (G\Box H)$,  $\gamma_P(G\Box H)\ptpd(G\Box H)$, and   the bounds in Theorem~\ref{Cartesian-thm}.  

\begin{figure}[h]
\begin{center}
\begin{tikzpicture}[scale=.7]
\begin{scope}[very thick, every node/.style={sloped,allow upside down}]

\vertex(A) at (0,5)[label=above:$w$] {};
\vertex(B) at (0, 3) [label=left:$u_1$]{};
\vertex(C) at (-0, 4)[label=left:$v_1$] {};
\vertex(D) at (1, 3) [label=right:$u_2$]{};
\vertex(E) at (1, 4) [label=right:$v_2$]{};
\vertex(F) at (2, 2) [label=above:$u_3$]{};
\vertex(G) at (3, 2)[label=above:$v_3$] {};
\vertex(H) at (-1, 2) [label=above:$u_8$]{};
\vertex(I) at (-2, 2)[label=above:$v_8$] {};
\vertex(J) at (-1, 1 )[label=below:$u_7$] {};
\vertex(K) at (-2, 1 )[label=below:$v_7$] {};
\vertex(L) at (2, 1 ) [label=below:$u_4$]{};
\vertex(M) at (3, 1 )[label=below:$u_4$] {};
\vertex(N) at (-0,0 )[label=left:$u_6$] {};
\vertex(O) at (-0,-1 ) [label=left:$v_6$]{};
\vertex(P) at (1,-0 ) [label=right:$u_5$]{};
\vertex(Q) at (1,-1 ) [label=right:$v_5$] {};

\draw(A) to (C);
\draw(C) to (B);
\draw(B) to (D);
\draw(D) to (E);
\draw(D) to (F);
\draw(F) to (L);
\draw(L) to (P);
\draw(P) to (N);
\draw(N) to (J);
\draw(J) to (H);
\draw(H) to (B);

\draw(F) to (G);
\draw(L) to (M);
\draw(P) to (Q);
\draw(N) to (O);
\draw(J) to (K);
\draw(H) to (I);

\end{scope}
\end{tikzpicture}
\end{center}
\caption{The graph $G$ in Example \ref{ex-cp<bd}.}\label{graphW}
\end{figure}
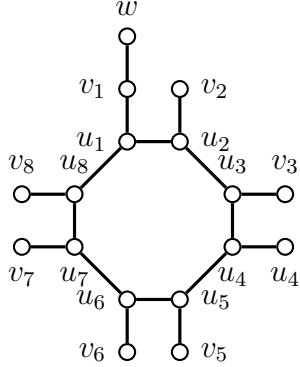
\begin{ex}\label{ex-cp<bd}Let $G$ be the graph in Figure~\ref{graphW} with vertex set $\{u_1, u_2, \ldots, u_8\}\cup\{v_1, v_2, \ldots, v_8\}\cup \{w\} $
where the induced subgraph on $\{u_1, u_2, \ldots, u_8\}$ is an 8-cycle, $v_i$ is adjacent to $u_i$ for $1\leq i\leq 8$, and $w$ is adjacent to $v_1$. Thus $G$ is an 8-cycle with a path of length 2 appended to $u_1$ and a leaf appended to $u_j$  for $2\leq j\leq 8$. Note that $G$ has 17 vertices. We denote the vertices of $G \Box P_2$ by $\{u_1, \dots, u_8,v_1,  \dots, v_8,w,u'_1, \dots, u'_8,$ $v'_1,  \dots, v'_8,w'\} $.

To show that $\thpdx(G \Box P_2)= 10$, note first that $\ppt(G \Box P_2,\{u_1,u_5,u_1',u_3',u_7'\})=2$, so $\thpdx(G \Box P_2)\le 10$.   To show the reverse inequality, we use \cite{sage} to compute $\gamma(G \Box P_2)=11$ and  $\pd(G\Box P_2)=3$. Since $\thpdx(G \Box P_2)\le 10<\gamma(G \Box P_2)$, we need consider only power dominating  sets  $S$ such that  $3\le |S|\le \frac {10} 2$ by Proposition \ref{r:comp-bd-thpdx}\eqref{r:comp-bd-thpdx-2}.   We compute    $\ppt(G\Box P_2,3)=7$ so $\pd(G \Box P_2) \ptpd(G \Box P_2)=\thpdx(G \Box P_2,3)=21$, and $\ppt(G\Box P_2,4)=3$  so $\thpdx(G \Box P_2,4)=12$. Since $5<\gamma(G \Box P_2)$, $\thpdx(G \Box P_2,5)\ge 10$.

   Since $\pd(G)=3$, $\ppt(G)=2 $, and $\gamma(G)=8$, we have $\thpdx(G)=6$ by Proposition \ref{r:comp-bd-thpdx}\eqref{r:comp-bd-thpdx-2}. Thus $\thpdx(G)|V(P_2)|=12$ and $\thpdx(P_2)|V(G)|=17$.
\end{ex}

As with upper bounds,  the structure of a Cartesian product gives additional lower bounds.  

\begin{obs}\label{o:lowerbd-cp} For connected graphs $G$ and $H$, $\thpdx(G \Box H) \geq \lc \frac{|V(G)||V(H)|}{\Delta(G) + \Delta(H) + 1} \rc$ by Corollary~\ref{t:lowerbd}  and the fact $V(G \Box H) = |V(G)||V(H)|$ and $\Delta(G \Box H) = \Delta(G) + \Delta(H)$. 
\end{obs}

We need a preliminary result about power dominating sets.   If $G \Box H$ is a Cartesian product of graphs $G$ and $H$ and $S\subset V(G\Box H)$, define the \emph{projection of $S$ onto $G$}, denoted by $S_G$, to be $S_G = \{ x : (x, y) \in S \mbox{ for some }y\in V(H) \}$.

\begin{prop} \label{cartesian-prod-proj}
Let  $G$ and $H$ be connected graphs and let $S$ be a power dominating set of $G \Box H$.  Then $S_G$ is a power dominating set of $G$.  Furthermore, $\ppt(G;S_G)\le \ppt(G\Box H;S)$. \end{prop}
\bpf
Let $S' = S_G \times V(H)$ and note that $S \subseteq S'$. Then $S'$ is a power dominating set of $G \Box H$ since $S$ is a power dominating set of $G \Box H$.  For a (propagating) set of forces, all forces for $S'$ in $G \Box H$ have the form $(x_1,y)\to (x_2,y)$, and $\rd(x,y)=\rd(x,z)$ for all $x\in V(G)$ and $y,z\in V(H)$. For $x\in V(G)$ and $y\in V(H)$,  note that $\rd(x)=\rd(x,y)$ starting with $S'_G=S_G$ in $G$ and $S'$ in $G \Box H$.  %
Thus, $S_G$ is a  power dominating set of $G$ and $\ppt(G;S_G)=\ppt(G\Box H;S')\le \ppt(G\Box H;S)$.
\epf

\begin{thm} \label{cartesian-prod}
For any connected graphs $G$ and $H$, \[\thpdx(G \Box H) \geq  \thpdx(G)\mbox{ and }\thpdx(G \Box H) \geq  \thpdx(H).\]
\end{thm}
\begin{proof}
 Choose a set $S$ such that $\thpdx(G \Box H; S) = \thpdx(G \Box H)$. Then $S_G$ is a power dominating set of $G$ and $\ptpd(G; S_G) \leq \ptpd(G \Box H; S)$ by Proposition \ref{cartesian-prod-proj}. Since  $|S_G| \leq |S|$,  \[\thpdx(G) \leq |S_G|\ptpd(G; S_G) \leq |S|\ptpd(G\Box H; S) = \thpdx(G \Box H).\] 
The proof that $\thpdx(G \Box H) \geq \thpdx(H)$ is similar.
\end{proof}

\subsection{Families having  $\thpdx(G\Box H)= \gamma(G\Box H)$}
In this section we show that the product power throttling number equals the domination number for  Cartesian products of complete graphs with complete graphs, paths with paths (grid graphs), paths with cycles, and cycles with cycles. %

\begin{prop} For $2\le n\le m$, $ \pd(K_n \Box K_m) =n-1$. For $1\le n\le m$, $ \thpdx(K_n \Box K_m) =\gamma(K_n \Box K_m)=n$. \end{prop}
\bpf  Let $1\le n\le m$.  %
The result $ \thpdx(K_1 \Box K_m) =1=\gamma(K_1 \Box K_m)$ is immediate, so assume $n\ge 2$.  Let $V(K_n \Box K_m)=\{(i,j):1\le i\le n, 1,\le j\le m\}$. 

Since $\{(i,1):i=1,\dots,n-1\}$ is a power dominating set, $ \pd(K_n \Box K_m) \le n-1$. We construct a power dominating set $S$ that is not a dominating set and show that $S$ must have at least $n-1$ vertices in order for step \eqref{zfproc} of the power domination process to take place.  Without loss of generality, $(1,1)\in S$ and a neighbor of $(1,1)$ performs the first zero force (the first force after the domination step).  Observe that $N((1,1))=\{(i,1):i=2,3,\dots,n\}\cup\{(1,j):j=1,3,\dots,m\}$.  Neighbors of the form $(1,j)$ all behave similarly, so suppose first $(1,2)$ performs the first zero force.  There is exactly one unobserved neighbor of $(1,2)$ after the domination step. Since  $\{(1,j):j=1,3,\dots,m\}\subset N[S]$, without loss of generality the first zero force is $(1,2) \to (2,2)$. This implies $\{(i,2):i=1,3,\dots,n\}\subset N[S]$ and $(2,2)\not\in  N[S]$.  Thus $(i,2)\not\in S$ for $i=1,\dots,n$, which implies  there exist $(i,j_i)\in S$ for $i=3,\dots,n$.  Thus $|S|\ge n-1$. If a neighbor of the form $(i,1)$ performs the first zero force, then $|S| \ge m-1$.  Thus $\pd(K_n \Box K_m) =n-1$.

Since  $\gamma(K_n \Box K_m)=n$ and  $\pd(K_n \Box K_m) =n-1\ge \frac n 2$, we have $ \thpdx(K_n \Box K_m) =\gamma(K_n \Box K_m)=n$ by Proposition \ref{r:comp-bd-thpdx}\eqref{r:comp-bd-thpdx-1}. \epf

\begin{prop}\label{completeprod}
Let $H$ be a connected graph of order $n$ and let $G = H \Box K_m$ with $m \ge \Delta(H)(n - 1)+1$.  Then $\thpdx(G) =n=\gamma(G)$.
\end{prop}
\begin{proof}
 Since $V(H)\x \{y\}$ is a dominating  set of $G$ for any $y\in V(K_m)$, we know $\gamma(G)\le n$, so
  $\thpdx(G) \leq \gamma(G)\le n$. %
It remains to show that $\thpdx(G) \ge n$. By Observation  \ref{o:lowerbd-cp}, we also know that $\thpdx(G) \geq \left \lceil \frac{nm}{m + \Delta(H)} \right \rceil$ since $\Delta(K_m) = m - 1$. Note that $\left \lceil \frac{nm}{m + \Delta(H)} \right \rceil \geq n$ since  $m \ge \Delta(H)(n - 1)+1$. Thus $\thpdx(G) \ge n$.
\end{proof}

Since %
$\Delta(C_n) = 2$ and $\Delta(P_n) = 2$, the next result follows immediately from Proposition \ref{completeprod}.

\begin{cor}
If $G = H \Box K_m$ with $H = C_n$ or $P_n$ and $m \ge 2n - 1$, then $ \thpdx(G) =n$. 
\end{cor}

Next we show that $\thpdx(P_n\square P_m)=\gamma(P_n\square P_m)$, $\thpdx(P_n\square C_m)=\gamma(P_n\square C_m)$, $\thpdx(C_n\square P_m)=\gamma(C_n\square P_m)$,  and $\thpdx(C_n\square C_m)=\gamma(C_n\square C_m)$  for  all $n \le m$.  The power domination number of a grid graph is known \cite{DH06}:  For $m\ge n\ge 1$,\vspace{-5pt}
\beq\pd(P_n\square P_m)= \begin{cases}
\lc \frac n 4\rc & \mbox{if }n\not\equiv 4\mod 8 \\
\lc \frac {n+1} 4\rc & \mbox{if  }n\equiv 4\mod 8 
\end{cases}.\label{eq:pd-grid}\vspace{-5pt}\eeq
The domination number is known exactly for only  certain values of $n$; a summary of results appears in \cite{ACIOP11grid} and are detailed later as used.  
Let $J_n=P_n$ or $C_n$ for $n\ge 3$ and $J_n=P_n$ for $n=1,2$. 
Note that  $P_n\square P_m$ is a spanning subgraph of $J_n\square J_m$, so $\gamma(J_n\square J_m)\leq \gamma(P_n\square P_m) $.  %

We orient $J_n\square J_m$ near a given vertex $x$ as a grid with $n$ rows and $m$ columns, and refer to the directions from $x$ as north, east, south, and west. When $J_n=C_n$, there is an additional edge between the nothernmost vertex and southernmost vertex of each column, and when $J_m=C_m$, there is an additional edge between the easternmost vertex and westernmost vertex of each row. 

Let $S$ be a power dominating set of $J_n\square J_m$ and let $\F$ be a set of forces of $S$. For each vertex $w$ in $P^{(2)}(S)$, $\F$ defines a forcing chain  $v_0\to v_1\to w$. Define the functions $f_1: P^{(2)}(S)\rightarrow P^{(1)}(S)$ and $f_0:  P^{(2)}(S)\rightarrow S$ by $f_1(w)=v_1$, and  by $f_0(w)=v_0$. 
By the definition of power domination, $f_1 $ is an injective function (but $f_0$ need not be injective). 
For $u\in S$, define  $Q_u= \{w\in P^{(2)}(S) : f_0(w)=u\}$. 
Limiting the size of $P^{(2)}(S)$ is a key idea for the proofs that $\thpdx(J_n\square J_m)=\gamma(J_n\square J_m)$.

\begin{prop} \label{three-of} Let $n, m\geq 4$ and let $S$ be a power dominating set of $J_n\square J_m$. There is a set of forces $\F$ of $S$ such that  $|Q_x|\leq 3$ for each $x\in S$. 
\end{prop}
\begin{proof} %
For any $x\in S$, $|Q_x|\leq 4$ since $\deg(x)\le 4$. %
Suppose that $|Q_x|=4$. We claim that the forces $x\to y\to w$  must occur in the same direction on the grid, e.g., if $y$ is the north neighbor of $x$, then $w$ is the north neighbor of $y$. Suppose not, e.g., $y$ is the north neighbor of $x$ and  $w$ is the west neighbor  of $y$.  Then  $w$ is a neighbor of the west neighbor of $x$, so the west neighbor of $x$ cannot perform a force in round 2, contradicting $|Q_x|=4$. The other directions are similar.  %
 Hence  the vertices in $Q_x$ must be  the four vertices that are distance 2 from $x$ in the four directions (the square vertices in Figure~\ref{four-d}); this applies only because $|Q_x|=4$. %

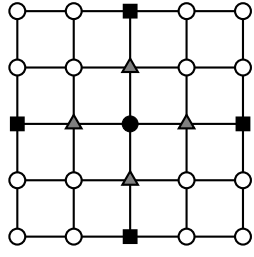
\begin{figure}[!h]
\centering
\begin{tikzpicture}[scale=.75,rotate=90]

\vertex (A) at (0,0) {};
\vertex (A) at (0,1){};
\vertex (A) at (1,0){};
\vertex (A) at (1,1){};
\node  at (2,0)[square]{};
\node at (2,1)[triangle]{};
\Bvertex (A) at (2,2){};
\node at (0,2)[square] {};
\node at (1,2) [triangle]{};
\vertex (A) at (3,0){};
\vertex (A) at (3,1){};
\node at (3,2)[triangle]{};
\vertex (A) at (3,3){};
\vertex (A) at (0,3){};
\vertex (A) at (1,3){};
\node at (2,3)[triangle]{};
\vertex (A) at (4,0){};
\vertex (A) at (4,1){};
\node at (4,2)[square]{};
\vertex (A) at (4,3){};
\vertex (A) at (0,4){};
\vertex (A) at (1,4){};
\node at (2,4)[square] {};
\vertex (A) at (3,4){};
\vertex (A) at (4,4){};
\begin{scope}[on background layer]
\draw [thick](0,0)--(4,0) ;
\draw[thick](0,1)--(4,1) ;
\draw[ thick](0,2)--(4,2) ;
\draw [thick](0,3)--(4,3) ;
\draw [thick](0,4)--(4,4) ;
\draw [thick](0,0)--(0,4) ;
\draw [thick](1,0)--(1,4) ;
\draw [thick](2,0)--(2,4) ;
\draw [thick](3,0)--(3,4) ;
\draw [thick](4,0)--(4,4) ;
\end{scope}
\end{tikzpicture}\\

\caption{The black circle vertex is $x\in S$, the triangle vertices are its neighbors and the square vertices are in $Q_x$. \label{four-d}}
\end{figure}

Let $x_N$ be the  north neighbor of $x$, let $x_W$ be the west neighbor of $x$, 
and let $x_{NN}$ be the north neighbor of $x_N$. In order to have $x_N=f_1(x_{NN})$, i.e., $x_N\to x_{NN}$ in round 2, the east and west neighbors of $x_N$, called $x_{NE}$ and $x_{NW}$,  must be observed in round 0 or round 1.  Suppose $x_{NW}$ is observed in round 1.  The south neighbor of $x_{NW}$ is $x_W$ and $x_W$ cannot observe $x_{NW}$  in round 1, nor can $x_N$ observe $x_{NW}$ in round 1.  Thus either the west neighbor $x_{NWW}$ or the north   neighbor $x_{NWN}$ of $x_{NW}$ observes $x_{NW}$ in round 1.  If $x_{NWW}\in S$, then the west neighbor of $x_W$ is observed in round 1, so $x_W$ cannot observe it in round 2, and similarly if $x_{NWN}\in S$ then $x_{NN}$ is observed in round 1.  So $x_{NW}$ cannot be observed in round 1 and thus $x_{NW}\in S$.
Then we can reassign the forcing chain $x\to x_N\to x_{NN}$ to $x_{NW}\to x_N\to x_{NN}$, obtaining $Q'_{x_{NW}}=Q_{x_{NW}}\cup\{x_{NN}\}$ and  $Q'_x=Q_x\setminus\{x_{NN}\}$. Since the south neighbor of $x_{NW}$ is $x_W$, which is still forced by $x$,  with the new assignment $|Q'_{x_{NW}}|\le 3$, and  $|Q'_x|\le 3$. The other directions are similar. 
\end{proof}

\begin{thm} Let  $n,m\ge 4$.  If $S$ is a power dominating set of $J_n\Box J_m$ that is not a dominating set, then  $\thpdx(J_n\Box J_m;S)\geq \lc\frac{nm}{4}\rc$ and $\thpdx(J_n\square J_m)=\gamma(J_n\square J_m)$.
\end{thm}
\begin{proof} Suppose $\ppt(J_n\Box J_m;S)\geq 2$ and let $t=\ppt(J_n\Box J_m;S)$. Then $|P^{(i+1)}(S)|\leq |P^{(1)}(S)|$ for all $i\geq 0$ by Remark~\ref{assign}. Since the maximum degree in $J_n\square J_m$ is 4,  $|P^{(1)}(S)|\leq 4|S|$.  
By Proposition~\ref{three-of}, there is an assignment of forcing chains so that for each vertex $x\in S$, $|\{w\in P^{(2)}(S) : f_0(w)= x\}|\leq 3$. Therefore, 
$|P^{(2)}(S)|\leq 3|S|$, and thus
\[nm= |V(J_n\square J_m)|=|S| +\sum _{i=1}^t |P^{(i)}(S)|\leq|S|(1+4 +3+ 4(t-2))=|S|(4t). \]
This implies  $\thpdx(J_n\Box J_m;S)=t|S|\geq \lc\frac{nm}{4}\rc$. This lower bound applies whenever $\ppt(J_n\Box J_m;S)\geq 2$.

Since $J_n\square J_m$ contains $P_n\square P_m$ as a spanning subgraph, $\gamma(P_n\Box P_m)\ge \gamma(J_n\square J_m)$.
 It is easily verified algebraically that $\frac{nm}{4}\ge  \frac{(n+2)(m+2)}{5} -4$ for $n,m\ge 8$, and Chang  showed in \cite{Ch92} that $\lf \frac{(n+2)(m+2)}{5}\rf-4 \ge  \gamma(P_n\Box P_m)$ for $n,m\ge 8$ (see \cite{ACIOP11grid}).  For $n=7$, it is known that $\gamma(P_n\square P_m)=\lf\frac{5m+3}{3} \rf$ \cite{ACIOP11grid}.   It is easily verified algebraically that $\frac{7m}{4}\ge \frac{5m+3}{3}$ for $m\ge 12$, and computationally that $\lc\frac{7m}{4}\rc\ge \lf\frac{5m+3}{3}\rf$ for $m=7,\dots,11$.  
  Thus $\thpdx(J_n\square J_m)=\gamma(J_n\square J_m)$ for $n,m\ge 7$.  

For $n=4$, it is known that   $\gamma(P_4\Box P_m)=m$ if $m\neq 5,6,9$ and $\gamma(P_4\Box P_m)=m+1$ if $m=5,6,9$ \  \cite{ACIOP11grid}. 
Since $\frac{4m}{4}=m$, $\thpdx(J_4\Box J_m)=\gamma(J_4\Box J_m) $ for $m\neq 5,6,9$.  For $G=C_4\Box J_m$ with $m=5,6,9$ or $G=P_4\Box C_m$ with $m=6$, $\gamma(G)=m$, so $\thpdx(G)=\gamma(G) $.   For the cases $G=P_4\Box P_m$ with $m=5,6, 9$ and $G=P_4\Box C_m$ with $m=5,9$, $\thpdx(G)=\gamma(G)$ has been verified computationally \cite{sage} and these values are listed in Table \ref{tab:JnJm}.

For $n=5$, it is known that    $\gamma(P_5\Box P_m)= \lf\frac{6m+8}{5} \rf$ if $m\ne 7$ and $\gamma(P_5\Box P_7)= 9$  \  \cite{ACIOP11grid}. 
It is easily verified algebraically that  $\frac{5m}{4}\ge \frac{6m+8}{5} $ for $m\ge 32$. Straightforward computations show that  $\lc\frac{5m}{4}\rc\ge \gamma(P_5\Box P_m)$ for $5\le m\le 31$ except $m=8$ and $m=12$.
 For $G=P_5\Box C_m$ or $G=C_5\Box J_m$ with $m=8,12$, $\lc\frac{5m}{4}\rc\ge\gamma(G)$, so $\thpdx(G)=\gamma(G) $.  For the case  $G=P_5\Box P_8$, $\thpdx(G)=\gamma(G)$ has been verified  \cite{sage} and this value is listed in Table \ref{tab:JnJm}, leaving only $P_5\Box P_{12}$ (this case is discussed at the end of the proof). 

For $n=6$, it is known that  $\gamma(P_6\Box P_m)=  \lf\frac{10m+12}{7} \rf$ if $m\not \equiv 1\mod 7$ and $\gamma(P_6\Box P_m)=  \lf\frac{10m+10}{7} \rf$ if $m \equiv 1\mod 7$  \  \cite{ACIOP11grid}. 
It is easily verified algebraically that $\frac{6m}{4}\ge \frac{10m+12}{7} $ for $m\ge 24$. Straightforward computations show that $\lc\frac{6m}{4}\rc\ge \gamma(P_6\Box P_m) $ for $6\le m\le 23$ except $m=6$ and $m=10$. For $G=P_6\Box C_m$ or $G=C_6\Box J_m$ with $m=6,10$, $\lc\frac{6m}{4}\rc\ge\gamma(G)$, so $\thpdx(G)=\gamma(G) $.   For the case  $G=P_6\Box P_6$, $\thpdx(G)=\gamma(G)$ has been verified  \cite{sage} and this value is listed in Table \ref{tab:JnJm}, leaving only $P_6\Box P_{10}$.

\begin{table}[h] 
\begin{center}
\caption{Table of values of $\thpdx(J_n\square J_m)$ and $\thpdx(J_n\square J_m)$ for selected $n$ and $m$.  \label{tab:JnJm}} 
{\small \begin{tabular}{|c|c|c|c|c|c|c|c |c|c|c|}
\hline
$J_n\square J_m$ & $\thpdx$ & $\gamma$ & & $J_n\square J_m$ & $\thpdx$ & $\gamma$& & $J_n\square J_m$ & $\thpdx$ & $\gamma$\\
\hline
$P_4\Box P_5$ & 6 & 6 & & $P_4\Box P_6$ & 7 & 7 & & $P_4\Box P_9$ & 10 & 10 \\
\hline
$P_4\Box C_5$ & 6 & 6  & & $P_4\Box C_9$ & 10 & 10  &&  &  & \\
\hline
$P_5\Box P_8$ &11  & 11 & & $P_6\square P_6$ & 10 &10  & &  &  &  \\
\hline
\end{tabular}}
 \end{center} \vspace{-10pt}
\end{table}

It remains to check $P_5\Box P_{12}$ and $P_6\Box P_{10}$, both of which have order 60, domination number 16 \cite{ACIOP11grid},  and power domination number 2 (see Equation \eqref{eq:pd-grid}).  Let $G$ denote one of these graphs.  Suppose $S$ is a power dominating set of $G$ that is not a dominating set of $G$, so $\ppt(G;S)\geq 2$.  Then $|S|\ppt(G;S)\geq \frac{mn}4=15$.  Since power propagation time is an integer, this implies $\ppt(G;S)\ge \lc\frac{15}{|S|}\rc$.  Thus
$\ppt(G,k)\ge \lc\frac{15}{k}\rc$ by using $k=|S|$.
Observe that  $k\ppt(G,k)\ge 16$ for $k\ge 8$, so consider    
$k\lc\frac{15}{k}\rc$ for $k=2,\dots,7$.  It is immediate that $k\lc\frac{15}{k}\rc\ge 16$ unless $k=3$ or $k=5$. We use \cite{sage} to compute
 $\ppt(P_5\Box P_{12},3)=6$, $\ppt(P_5\Box P_{12},5)=4$, $\ppt(P_6\Box P_{10},3)=9$, and $\ppt(P_6\Box P_{10},5)=5$. 
 This completes the proof.
  \end{proof}
 
The only remaining cases  are $P_2\square J_m$ and $J_3\square J_m$, which are handled in the next theorem.

\begin{thm} For $m\geq 2$, $\thpdx(P_2\square J_m) = \gamma(P_2\square J_m)$,  
and for $m\geq 3$, $\thpdx(J_3\square J_m) $ $= \gamma(J_3\square J_m)$.
\end{thm}

\begin{proof}  It is known that $\gamma(P_2\square P_m) = \lf \frac{m+2}{2}\rf$ 
\cite{ACIOP11grid}. Since $P_2\square C_m$ contains $P_2\square P_m$ as a spanning subgraph, $\lf 
\frac{m+2}{2}\rf\geq \gamma(P_2\square C_m)$. 
We show that if $S$ is a power dominating set of $P_2\Box J_m$ that is not a dominating set, then  $\thpdx(P_2\Box J_m;S)\geq \lc\frac{2m}{3}\rc$. It is straightforward to check that $\lc\frac{2m}{3}\rc\geq  \lf \frac{m+2}{2}\rf$ for $m\geq 2$, which then implies that $\thpdx(P_2\square J_m)=\gamma(P_2\square J_m)$.
For $x\in S$, denote the north, east, south, and west neighbors of  $x$ by $x_N, x_E, x_S$ and $x_W$.

Suppose $\ppt(P_2\Box J_m;S)\geq 2$ and let $t=\ppt(P_2\Box J_m;S)$. Then $|P^{(i+1)}(S)|$ $\leq |P^{(1)}(S)|$ for all $i\geq 0$ by Remark~\ref{assign}. Choose a set of forces $\F$ of $S$, and for $x\in S$ recall that  $Q_x= \{w\in P^{(2)}(S) : f_0(w)=x\}$.  Since the maximum degree in $P_2\square J_m$ is 3,  $|P^{(1)}(S)|\leq 3|S|$ and $|Q_x|\leq 3$ for $x\in S$.  Suppose $x\in S$ is on the bottom row of $P_2\Box J_m$. %
If $x_N$ forces to the east or west in round 2,  then the neighbor of $x$ in the same direction  cannot force  in round 2.   
Therefore, $|Q_x|\le 2$, $|P^{(2)}(S)|\leq 2|S|$, and thus
\[2m=|V(P_2\square J_m)|= |S| +\sum _{i=1}^t |P^{(i)}(S)|\leq|S|(1+3 +2+ 3(t-2))=|S|(3t). \]
This implies  $\thpdx(P_2\Box J_m;S)=t|S|\geq \lc\frac{2m}{3}\rc$.

It is known that $\gamma(P_3\square P_m) =  \lf \frac{3m+4}{4}\rf$ \cite{ACIOP11grid}, and thus $\gamma(J_3\square J_m)\leq \lf \frac{3m+4}{4}\rf $. 
 We show that  if $S$ is a power dominating set of $J_3\Box J_m$ that is not a dominating set, then  $\thpdx(J_3\Box J_m;S)\geq\lf \frac{3m+4}{4}\rf$ and therefore $\thpdx(J_3\square J_m)=\gamma(J_3\square J_m)$.
 Suppose $\ppt(J_3\Box J_m;S)\geq 2$ and let $t=\ppt(J_3\Box J_m;S)$.  Since the maximum degree in $J_3\square J_m$ is 4,  $|P^{(1)}(S)|\leq 4|S|$. 

Choose a set of forces $\F$ of $S$ such that  for each $y\in P^{(1)}(S)$, if $y$ is  adjacent to a vertex in $S$ along a row edge, i.e., $y$ is a {row-neighbor} of a vertex in $S$, then $y$ is forced by one of its row-neighbors in $S$. 
For $x\in S$, if $\deg(x) =3$, then $|Q_x|\leq 3 $. Let $x\in S$ with $\deg(x)=4$; we show this implies $|Q_x|\leq 2$. %
 If for both $x_N$ and $x_S$,  this vertex is not forced by $x$ or does not force in round 2, then $|Q_x|\leq 2$ is immediate. %

So suppose that $x\to x_N$ and $x_N$ forces in round 2. Then $x_N$ cannot force to the north, because if $J_3=P_3$, then there is no north neighbor of $x_N$, and if $J_3=C_3$, then the north neighbor of $x_N$ is  $x_S$. Without loss of generality, suppose $x_N$ forces its west neighbor $x_{NW}$ in round 2.   
This implies  $x_W$ cannot force in round 2.
In order for $x_N\to x_{NW}$ in round 2, the east neighbor  $x_{NE}$ of $x_N$  must have $\rd(x_{NE})=0$ or $\rd(x_{NE})=1$. If $\rd(x_{NE})=0$, then $x_{NE}\in S$ and  $x_N$ is a row-neighbor of $x_{NE}$, so $x_N$ would not be forced by $x$.  

Thus  $\rd(x_{NE})=1$, so $x_{NE}$ is adjacent to a vertex $u$ in $S$. We show this implies another neighbor of $x$ cannot contribute to $Q_x$, and thus $|Q_x|\le 2$. 
If $u$ is the south neighbor of $x_{NE}$, then $u=x_E$, so  $x_E$ does not contribute to $Q_x$. 
If $J_3=C_3$ and $u$ is the north neighbor of $x_{NE}$, then  $u$ is the east neighbor of  $x_S$, so  $x_S$ cannot be forced by $x$ (because it is a row-neighbor of $u\in S$); thus  $x_S$ does not contribute to $Q_x$. 
If $u$ is the east neighbor of $x_{NE}$, then  $x_E$ cannot force east in round 2, because $u$ is adjacent to the east neighbor of  $x_E$.  If $x_E$ forces south in round 2, then $x_S$ cannot force in round 2.
 Hence $|Q_x|\leq 2$.  
 
 Since $3+3=4+2=6$, we have $|P^{(1)}(S)| +|P^{(2)}(S)|\leq 6|S|$, and thus
\vspace{-4pt}
\[3m=|V(J_3\square J_m)|= |S| +\sum _{i=1}^t |P^{(i)}(S)|\leq|S|(1+6+ 4(t-2))=|S|(4t-1). \]
This implies  $\thpdx(J_3\Box J_m;S)=t|S|\geq t \lc\frac{3m}{4t-1}\rc$. 
Then $t \lc\frac{3m}{4t-1}\rc\ge t\frac{3m}{4t-1}  > t \frac{3m}{4t} \ge \lf \frac{3m}{4}\rf$. Since the first and last terms are integers,  $t \lc\frac{3m}{4t-1}\rc\geq \lf \frac{3m}{4}\rf+1 =\lf \frac{3m+4}{4} \rf$.
\end{proof}

We conclude with a corollary that summarizes the situation for Cartesian products of connected graphs in which each factor graph has degree at most 2. 

\begin{cor} For all $n,m\geq 1$,   $\thpdx(J_n\square J_m)=\gamma(J_n\square J_m)$, where $J_k=P_k$ or $J_k=C_k$ for $k\ge 3$ and $J_k=P_k$ for $k=1,2$.
\end{cor}%

\section*{Acknowledgements}

The authors gratefully acknowledge the support and hospitality of the Institute for Mathematics and its Applications (IMA) during the Workshop for Women in Graph Theory and Applications (WIGA), where this collaboration was initiated and the work described in this article began; IMA funds were used to support WIGA, but no NSF funds awarded to IMA were used. 

We are grateful for the support provided by the AWM ADVANCE Research Communities Program, funded by NSF-HDR-1500481, Career Advancement for Women Through Research-Focused Networks.

The work of A. Trenk was partially supported by a grant from the Simons Foundation (\#426725).
The work of C. Mayer was partially supported by Sandia National Laboratories.  Sandia National Laboratories is a multimission laboratory managed and operated by National Technology \& Engineering Solutions of Sandia, LLC, a wholly owned subsidiary of Honeywell International Inc., for the U.S. Department of Energy's National Nuclear Security Administration under contract DE-NA0003525. This paper describes objective technical results and analysis. Any subjective views or opinions that might be expressed in the paper do not necessarily represent the views of the U.S. Department of Energy or the United States Government.

\end{document}